\documentclass[12pt,twoside,reqno]{amsart}
\usepackage{amsmath,amsfonts,amsthm,amssymb}
\usepackage{graphicx,overpic}%,psfrag
\usepackage{bm}
\usepackage[all,cmtip,line]{xy}
\usepackage{array,setspace}
\usepackage{color}
\numberwithin{equation}{section}

\usepackage{graphics}
\usepackage{epsfig}
\newtheorem{claim}{\bf \t}[part]

\newtheorem{theorem}{Theorem}[section]

\newtheorem{lemma}{Lemma}[section]
\newtheorem{proposition}[theorem]{Proposition}
\newtheorem{remark}{Remark}[section]

\def\XXint #1#2#3{{\setbox 0=\hbox {$#1{#2#3}{\int }$}
		\vcenter {\hbox {$#2#3$}}\kern -.5\wd 0}}

\def\t#1{\tilde{#1}} %new variables

\def\R{{\mathbb{R}}}

\def\ba{\begin{aligned}}
	\def\ea{\end{aligned}}
\def\be{\begin{equation}}
\def\ee{\end{equation}}

\def\bes{\begin{mysubequations}}
	\def\ees{\end{mysubequations}}

\def\lb#1{\|#1\|_{L^2}}

\def\ol{{\Omega_L}}

\def\slb{{S_L^+}}
\def\cl#1{\overline{#1}}
\def\ra{\rightarrow}

\def\g{\nabla}
\def\a{\alpha}

\def\fai{\phi}
\def\O{\Omega}
\def\f{\frac}
\def\p{\partial}

\def\bt{\begin{theorem}}
	\def\et{\end{theorem}}
\def\ba{\begin{array}}
	\def\ea{\end{array}}
\def\bl{\begin{lemma}}
	\def\el{\end{lemma}}

\def\bes{\begin{eqnarray}}
\def\ees{\end{eqnarray}}
\begin{document}

\title[Low Mach Number Limit of Steady Euler Flows in Multi-Dimensional Nozzles]
{Low Mach Number Limit of Steady Euler Flows in Multi-Dimensional Nozzles}

\author{Mingjie Li}
\address{M.J. Li, College of Science, Minzu University of China, Beijing 100081, P. R. China}
\email{lmjmath@163.com}

	\author{Tian-Yi Wang}
	\address{T.-Y. Wang, Department of Mathematics, School of Science, Wuhan University of Technology,
		Wuhan, Hubei 430070, P. R. China;
		Gran Sasso Science Institute, viale Francesco Crispi, 7, 67100 L'Aquila, Italy}
	\email{tianyiwang@whut.edu.cn; tian-yi.wang@gssi.infn.it; wangtianyi@amss.ac.cn}
	
	\author{Wei Xiang}
	\address{W. Xiang, City University of Hong Kong, Kowloon Tong, Hong Kong, P. R. China}
	\email{weixiang@cityu.edu.hk}

\date{\today}

\begin{abstract}%\textcolor{red}{We need to revise this part.}
In this paper, we consider the steady irrotational Euler flows in multidimensional nozzles. The first rigorous proof on the existence and uniqueness of the incompressible flow is provided. Then, we justify the corresponding low Mach number limit,  which is the first result of the low Mach number limit on the steady Euler flows. We establish several uniform estimates, which does not depend on the Mach number, to validate the convergence of the compressible flow with extra force to the corresponding incompressible flow, which is free from the extra force effect, as the Mach number goes to zero. The limit is on the H\"{o}lder space and is unique. 
Moreover, the convergence rate is of order $\varepsilon^2$,  which is higher than the ones in the previous results on the low Mach number limit for the unsteady flow.
\end{abstract}

\keywords{
Multidimensional,
low Mach number limit,
steady flow,
homentropic Euler equations,
%isentropic,
%rotation,
%compactness framework,
convergence rate
%approximate solutions
}
\subjclass[2010]{
35Q31; %Euler Equation
35L65; %Conservation law
76N15; %Gas dynamics, general
35B40; %Asymptotic behavior of solution
%35D30% Weak solution
}

\maketitle

\section{Introduction}

Incompressible and compressible Euler equations are fundamental equations in fluid dynamics, which describe the motion of two different objects, for example, water and gas. %fluid flows are two  differs from compressible fluid flow in that the former one of the equations is to restrain that the flow is divergence-free. 
As one of the important topic in the mathematical theory of fluid dynamics, the incompressible limit is devoted to building a bridge to fill the gap between those two different type of fluids by determining in which sense that the compressible flows tend to the incompressible ones as the compressibility parameter tends to zero. 

One of the typical system to describe compressible flow is the steady homentropic Euler equations. It reads

\begin{eqnarray}\label{OrE}
\begin{cases}
	\mbox{div}(\rho u)=0,\\
	\mbox{div}(\rho u\otimes u)+\nabla p=\rho F,%\\
%	\mbox{curl} u =0,
\end{cases}
\end{eqnarray}
where $x=(x_1, \cdots, x_n)\in \R^n$, for $n\ge 2$, $u=(u_1,\cdots,u_n)\in \R^n$ is the fluid velocity, while $\rho$, $p$, and $F$ represent the density, pressure, and extra force respectively. The pressure is a function of density with:
\begin{equation}
p:=\frac{\tilde{p}(\rho)-\tilde{p}(1)}{\varepsilon^2},
\end{equation}
where $\varepsilon>0$ is the compressibility parameter as introduced in \cite{Schochet}. As the homentropic flow, we require
\begin{equation}\label{conditiononpressure}
\tilde{p}'(\rho)>0,  \quad 2\tilde{p}'(\rho) + \rho \tilde{p}''(\rho) > 0 \qquad \mbox{for $\rho>0$}.
\end{equation}
We remark that condition \eqref{conditiononpressure} holds for the flows which satisfy the polytropic gases that  $\tilde{p}= \rho^\gamma$ with $\gamma>1$,
and the isothermal flows that $\tilde{p}= \rho$.
The sound speed of the flow is $$c:=\sqrt{p'(\rho)}=\frac{\sqrt{\tilde{p}'(\rho)}}{\varepsilon},$$
and the Mach number is  defined as $$M:=\frac{|u|}{c}=\frac{\varepsilon |u|}{\sqrt{\tilde{p}'(\rho)}},$$ 
where
$$
|u|:=\Big(\sum_{i=1}^n u_i^2\Big)^{1/2}
$$
is the flow speed. The flow is subsonic when $M<1$, sonic when $M=1$, and supersonic when $M>1$.
%while the $M=1$ means the flow is locally sonic. Otherwise, $M>1$ implies flow is supersonic.

Generally speaking, there are two process for deriving the incompressible fluid models from the respective compressible ones: one is that the compressible parameter $\varepsilon$ goes to zero, which is called as the low Mach number limit; and the other one is that the adiabatic exponent $\gamma$ goes to the infinity with the pressure defined as $p=\rho^\gamma$, see \cite{CHWX,  Masmoudi2}.

The first theory of the low Mach number limit is due to Janzen and Rayleigh (see \cite[Sect. 47]{Schiffer}, \cite{VD}), 
in which they are concerned with the steady irrotational flow. Their method of the expansion of solutions in power with respect to the Mach number was applied both as a computational tool and as a tool for the proof of existence of solutions.  
Klainerman and Majda \cite{KM1, KM2} proved the convergence of compressible flow to the incompressible flow by directly deriving estimates of solutions of the partial differential equations in the scaled form (also see Ebin \cite{Edin}). In particular, they established the incompressible limit
of smooth local solutions of the Euler equations (and the Navier-Stokes equations)
for compressible fluids with well-prepared initial data, {\it i.e.}, some smallness assumption on the divergence of initial velocity.
By using the fast decay property of acoustic waves, Ukai \cite{Ukai} verified the low Mach number limit for the general data. The exterior domain cases were considered in  \cite{Isozaki1}. 
%In \cite{Schochet1}, Schochet showed the limit of the full compressible Euler equations to the incompressible inhomogeneous Euler equations in bounded domains for smooth local solutions with well-prepared initial data. 
% He \cite{Schochet}  further studied the fast singular limit for general hyperbolic partial differential equations.  
The major breakthrough on the general initial data is due to M\'{e}tivier and Schochet \cite{MS}, in which they proved the low Mach number limit of the full Euler equations in the whole space by an elegant convergence lemma on acoustic waves. Later, Alazard \cite{alazard1} extended the result to the exterior domain problem. %The one dimension spatial periodic case on the full compressible Euler equations was considered in \cite{MS1}.  
For the one dimensional Euler equations, the low Mach number limit has been proved under the $B.V.$ space in \cite{CCZ}.  For other related fluid models and problems, see \cite{DBDL, CDX,Feireisl-N, JJL3, JJLX, Lions P.-L.01, Masmoudi, MS,QX,Schochet} and the references therein.

One of classical problems on the steady flows is the infinitely long nozzle problem. Let  $\Omega\subset\mathbb{R}^n$ be an infinitely long nozzle, which is homomorphism on the unit cylinder $\textbf{C} = B(0,1)\times\mathbb{R}$ in $\mathbb{R}^n$. The compressible fluid fills in the region $\Omega$.   At the boundary $\partial\Omega$, the flow satisfies the slip boundary condition:
\begin{equation}\label{slipcondition}
u \cdot \textbf{n} =0 \qquad \mbox{on} ~\partial\Omega,
\end{equation}
where  $\textbf{n}$ is the unit outward normal to the region $\Omega$.  Due to $(\ref{OrE})_1$ and \eqref{slipcondition}, one can obtain the fixed mass flux property: on the arbitrary cross
section of the nozzle $S_0$
\begin{equation}\label{massfluxcondition}
\int_{S_0} \textbf{l}\cdot \rho  u ds=:m,
\end{equation}
where  $\textbf{l}$ is the unit outer
normal of the domain $S_0$. $m$ is called the mass flux. See Figure \ref{F1.1}.
\begin{figure}
	\centering
	\includegraphics[width=0.6\linewidth]{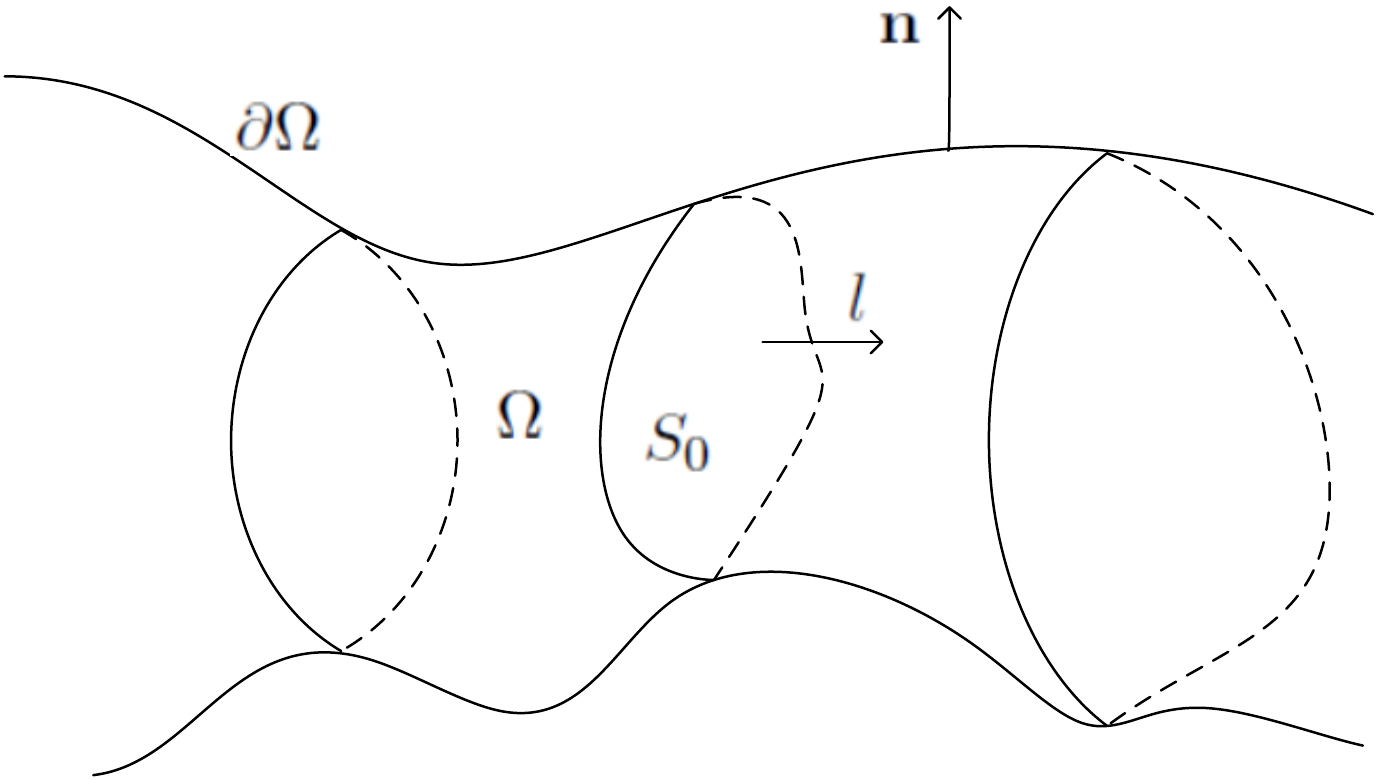}
	\caption{Infinitely long nozzle}
	\label{F1.1}
	\label{fig:boundedinfinitylongnozzle}
\end{figure}

Formally, if $|u|$ is bounded and $\sqrt{\tilde{p}'(\rho)}$ does not vanish, $\varepsilon$ is the leading term of Mach number $M$. For this reason, the limit $\varepsilon\rightarrow0$ is called the low Mach number limit\cite{KM1, KM2}. So, for the low Mach number limit, one should start from the case with sufficiently small Mach number, in which the flow is subsonic. 

For the compressible flow, the mathematical theory on global subsonic flow in an infinitely long nozzle of various cross-section was formulated by Bers  \cite{Bers4} in 1958. Then the first rigorous proof for the irrotational flow was achieved by Xie and Xin \cite{Xin1} by introducing the stream function. Later, they extended it to the  3D axis-asymmetric case  in \cite{Xin2}.  The theorem for general infinitely long nozzle
in $\R^n, n\geq 2$ was completed in Du-Xin-Yan \cite{Du-Yan-Xin}, while the result was extended in \cite{Gu-Wang1} to the extra force case.  
%The largely open nozzle case was proved by Liu and Yuan in \cite{Liu-Yian}.
Besides the infinitely long nozzle problem, we also would like to mention the study of the other classical problem: the airfoil problem. %Frankl-Keldysh \cite{Frankl},  
Shiffman \cite{Shiffman2},  Bers \cite{Bers1, Bers3}, and Finn-Gilbarg \cite{Finn1} considered the two-dimensional irrotational subsonic flow. 
Finn and Gilbarg \cite{Gilbarg1} got the first result for three-dimensional subsonic flow past
an obstacle under some restrictions on the Mach number.
Then Dong \cite{Dong1} and Dong-Ou \cite{Dong2} extended these results to the case when Mach number $M < 1$ for arbitrarily dimensional case, while the case with conservative force effect was considered in \cite{Gu-Wang2}. The respective subsonic-sonic flow was considered in \cite{Huang-Wang-Wang}. For the rotational subsonic flows, one can refer \cite{Chen-Du-Xie-Xin,Chen-Huang-Wang,CHWX,CHWX1,DWX,DXX,Xin4}. %On the other hand, as far as we know, up to now there is no result on the incompressible flow on the bounded nozzle problem.

The expected corresponding homogeneous incompressible Euler equations as $\varepsilon\rightarrow0$ are written as:

\begin{equation}\label{ICHE}
\begin{cases}
\mbox{div}\, u=0,\\
\mbox{div}\,(u\otimes u)+ \nabla p=F,
\end{cases}
\end{equation}
where $u=( u_1,\cdots, u_n)$ and $p$ represents the velocity and pressure, respectively, while density $\rho\equiv1$. 

For the problem of the incompressible flow in an infinitely long nozzle, the flow also satisfies the slip boundary condition \eqref{slipcondition}. Furthermore, from $(\ref{ICHE})_1$, the fixed mass flux property \eqref{massfluxcondition} holds with $\rho\equiv1$.

As far as we know, up to now there is no mathematical result on the multidimensional incompressible flow in an infinitely long nozzle. Hence, we need to develop the methods to obtain the first results on the multidimensional incompressible flow in an infinitely long nozzle.  We remark, unlike the airfoil problem, the solutions of infinitely long nozzle problem are not expected to decay to the given states in general, which means the variational approach could not applied directly. 
So we need to introduce approximate problems for both the incompressible and compressible case, and then to show the localized uniform estimates such as the average lemma and the higher order estimates which do not depend on the Mach number, which is a singular parameter in the low Mach number limit. All the uniform estimates are new. Moreover, to avoid the singularity arising from the sufficiently small Mach number, we need to introduce the elliptic cut-off carefully, which is different from the ones in \cite{Du-Yan-Xin, Gu-Wang1}.  
%\textcolor{red}{\bf As far as we know, up to now there is no result on the incompressible flow on the bounded nozzle problem. Here, we need to list the different and difficulty of the problem. I just list something we need to be mention: first result of incompressible flow in the nozzle, which is differential from the airfoil one; The uniform estimates are new; maybe something else from the low Mach limit in the airfoil paper.  Also, for the two papers on the conservative force \cite{Gu-Wang1,Gu-Wang2}.}

The rest of this paper is organized as follows. In Section 2, we formulate the problems mathematically and state the main theorems. Next, we introduce approximate problems and apply the variational approach to solve them in Section 3. In Section 4, uniform estimates and then the existence of modified flows are proven. Section 5 prove the uniqueness of  modified flows, and complete the existence and uniqueness of the incompressible flow (Theorem \ref{MainT1}). Finally, in Section 6, we complete the proof of the low Mach number limit (Theorem \ref{MainT2}). 
 %\textcolor{red}{The rest of this paper is organized as follows. In Section 2, we establish the formulation of the problem and state the main theorem. Next, we clarify the mathematical setting and introduce the cut-off by modifying density function in Section 3. For the modified problem, the variation formulation is used to constructing the solution in Section 4 ,and the local average estimate is proven also. In Section 5, the uniqueness  of  modified flows is proved. Finally, in Section 6, we complete the proof by the varied Bers skill and subsonic-sonic compactness argument.}

\section{Formulation of the problem and main theorems}
In this section we will formulate the problem concerned mathematically and introduce the main theorems of this paper.

First, we will give the basic assumptions on the multidimensional nozzle domain $\Omega$:
there exists an invertible
$C^{2,\alpha}$ map $T: \bar{\Omega}\rightarrow\mathbf{\bar{C}}:x\rightarrow
y$ which satisfies that
\begin{eqnarray}\label{H1}
\begin{cases}
T(\partial\Omega)=\partial \mathbf{\bar{C}},\\
\mbox{For any}\ k\in\mathbb{R},\
T(\Omega\cap\{x_n=k\})=B(0,1)\times\{y_n=k\},\\
\|T\|_{C^{2,\alpha}}+\|T^{-1}\|_{C^{2,\alpha}}\leq K<\infty,\\
\end{cases}
\end{eqnarray}
where $K$ is a uniform constant, $\mathbf{\bar{C}}=B(0,1)\times(-\infty,+\infty)$
is the cylinder in $\mathbb{R}^n$ with $B(0,1)$ being the unit ball in
$\mathbb{R}^{n-1}$ centering at the origin, $x_n$ is the
axial coordinate and $x'=(x_1,\cdots,x_{n-1})\in
\mathbb{R}^{n-1}$.

\begin{remark}
It is noticeable that there is no asymptotic restriction on $T$ as $x_n\rightarrow\pm\infty$, which implies $T$ could even be periodic respect to $x_n$. 
\end{remark}

For both the incompressible flow and the compressible flow, the boundary condition \eqref{slipcondition} and mass flux condition \eqref{massfluxcondition} hold.

Now, for any given mass flux $m$, we can introduce the problem in the multidimensional infinitely long nozzle mathematically for both the incompressible and compressible cases. 

\textbf{Problem ($m$)}. Let $n\geq 2$. Find functions $(\rho, u, p)$, which satisfies \eqref{OrE} or \eqref{ICHE}, with slip boundary condition \eqref{slipcondition} and mass flux condition \eqref{massfluxcondition}.
 
Then we will study \textbf{Problem ($m$)} for the incompressible and compressible cases separately.
 
 %It is noticeable that for the incompressible case, the density $\rho\equiv1$.
 
 %In this paper, we have following to two main theorem.

\subsection{Problem $(m)$ for the incompressible case}
Let us consider the incompressible case first. The steady irrotational incompressible Euler flow is governed by the following equations: 
\begin{equation}\label{CHEB}
\begin{cases}
\mbox{div}\, \bar{u}=0,\\
\mbox{div}\,(\bar{u}\otimes \bar{u})+ \nabla \bar{p}=F,\\
\mbox{curl}\, \bar{u}=0.
\end{cases}
\end{equation}

%Here, the density $\bar{\rho}\equiv1$. On the other hand, 
Here, the conservative force $F$ can be written as $F=\nabla\phi$.
By $(\ref{CHEB})_3$, we have the Bernoulli law for the incompressible irrotational case that:
\begin{equation}
\nabla\left( \frac{|\bar{u}^2|}{2}+\bar{p}-\phi\right) =0,
\end{equation}
which equals to:
\begin{equation}\label{pbar}
\bar{p}=\phi-\frac{|\bar{u}^2|}{2}
\end{equation}
up to a constant.
%It means the pressure satisfies that
%\begin{equation}
%\bar{p}=\frac{B-|\bar{u}|^2}{2}%=\frac{B-|\nabla\bar{\varphi}|^2}{2}
%\end{equation}
%for some constant $B$. 

Now we can introduce the following theorem for the incompressible case.
\begin{theorem}\label{MainT1}
	For any fixed $m>0$,  there exists a unique solution $(\bar{\rho}\equiv1, \bar{u}, \bar{p})$ of \textbf{Problem (m)} corresponding to equations \eqref{CHEB}, with $\bar{u} \in (C^{1,\alpha}(\Omega))^{n}$. Moreover, suppose 
	\begin{equation}
	\label{psicondition1}
	\phi\in L^\infty(\Omega)\qquad\text{and}\qquad \nabla\phi\in L^q_{loc}(\Omega)\qquad\text{for}\qquad q>n.
	\end{equation}
	Then $\bar{p} \in C^{\alpha}(\Omega)$.
\end{theorem}

We remark that for the incompressible case, the velocity $\bar{u}$ is solved by solving the potential function $\bar{\varphi}$, which does not depend on $\phi$. It is different from the compressible case.

\subsection{Problem $(m)$ for the compressible case and the low Mach number limit} Now let us consider \textbf{Problem $(m)$} for the compressible case and the low Mach number limit.
The irrotational compressible Euler flow with low Mach number is governed by the following equations
\begin{equation}\label{CISEV}
\begin{cases}
\mbox{div}\,(\rho^{\varepsilon} u^{\varepsilon})=0,\\
\mbox{div}\,(\rho^{\varepsilon} u^{\varepsilon}\otimes u^{\varepsilon})+ \nabla p^{\varepsilon} =\rho^\varepsilon F,\\
\mbox{curl}\, u^{\varepsilon}=0,
\end{cases}
\end{equation}
where the conservative force $F=\nabla\phi$ and $
p^{\varepsilon}=p^{(\varepsilon)}(\rho^{\varepsilon})=\frac{\tilde{p}(\rho^{\varepsilon})-\tilde{p}(1)}{\varepsilon^2}$. The Mach number is defined as $M^\varepsilon=\frac{|u^\varepsilon|}{\sqrt{(p^{(\varepsilon)})'(\rho^\varepsilon)}}=\frac{\varepsilon|u^\varepsilon|}{\sqrt{\tilde{p}'(\rho^\varepsilon)}}$.

By $(\ref{CISEV})_3$, we have the following Bernoulli law that
\begin{equation}
\nabla\left(\frac{|u^\varepsilon|^2}{2}+\int^{\rho^{\varepsilon}}_1 \frac{(p^{(\varepsilon)})'(s)}{s} ds-\phi
\right)=0.
\end{equation}
which is equivalent to 
\begin{equation}\label{BLaw-1}
\frac{|u^\varepsilon|^2}{2}+\int^{\rho^{\varepsilon}}_1 \frac{(p^{(\varepsilon)})'(s)}{s} ds=\phi,
\end{equation}
up to a constant. Here, without loss of generality, we assume $0\leq\phi\leq\phi^\star$.
Moreover, let us introduce the rescaled enthalpy function $\tilde{h}$, which satisfies:
\begin{equation*}
\tilde{h}'(\rho)=\frac{\tilde{p}'(\rho)}{\rho}.%\qquad \mbox{and}\qquad \tilde{h}(1)=1.
\end{equation*} 
%and 
%\begin{equation}
%\tilde{h}(1)=1.
%\end{equation}
Then, \eqref{BLaw-1} becomes
\begin{equation}
%\frac{|u^\varepsilon|^2}{2}+h^{(\varepsilon)}(\rho^\varepsilon)=B+h^{(\varepsilon)}(1),
\frac{|u^\varepsilon|^2}{2}+h^{(\varepsilon)}(\rho^\varepsilon)-h^{(\varepsilon)}(1)
=\phi,
\end{equation}
where $h^{(\varepsilon)}(\rho)=\varepsilon^{-2} \tilde{h}(\rho)$. By \eqref{conditiononpressure}, we can see that $\tilde{h}(\rho)$ is a strictly increasing function with respect to $\rho$, so does $h^{(\varepsilon)}(\rho^\varepsilon)$.

Let 
$$
\tilde{H}(\rho):=\frac{\tilde{p}'(\rho)}{2}+\tilde{h}(\rho).
$$

Then for each fixed $\varepsilon$, we can introduce the critical density %Firstly, let us introduce the $\tilde{H}(\rho):=\frac{\tilde{p}'(\rho)}{2}+\tilde{h}(\rho)$. 
%For  the critical density 
$\rho^\varepsilon_{cr}$ which satisfies that
\begin{equation}
\frac{1}{\varepsilon^2}\tilde{H}(\rho^\varepsilon_{cr})=\phi+\frac{1}{\varepsilon^2}\tilde{h}(1).
\end{equation}

So 
\begin{equation}
\rho^\varepsilon_{cr}(\phi)=\tilde{H}^{-1}\left(\tilde{h}(1)+\varepsilon^2 \phi\right)
\end{equation}
and the critical speed is 
\begin{equation}
q^\varepsilon_{cr}(\phi)=\frac{1}{\varepsilon}\sqrt{\tilde{p}'\circ\tilde{H}^{-1}\left(\tilde{h}(1)+\varepsilon^2 \phi\right)}.
\end{equation}
Obviously, the critical speed $q^{\varepsilon}_{cr}(\phi)$ will go to the infinity as $\varepsilon$ goes to zero.
%where $h_\varepsilon(\rho)=\frac{1}{\varepsilon^2}\int_1^\rho\frac{\tilde{p}'(s)}{s} ds$.

It is easy to see that $|u^\varepsilon|< q^\varepsilon_{cr}(\phi)$ holds if and only if the flow is subsonic, \emph{i.e.}, $M^\varepsilon(\phi) < 1$. Similarly,  %form the conditions on the pressure \eqref{conditiononpressure}, 
for each $0<\theta<1$, there exists $q^\varepsilon_{\theta}(\phi)$ such that $|u^\varepsilon|\leq q^\varepsilon_{\theta}(\phi)$ holds if and only if $M^\varepsilon(\phi)\leq \theta$. Moreover, $q^{\varepsilon}_{\theta}(\phi)$ is monotonically increasing with respect to $\theta\in(0,1)$. In addition, for $\varepsilon>0$, both $\varepsilon q^\varepsilon_{cr}(\phi)$ and  $\varepsilon q^\varepsilon_{\theta}(\phi)$ are uniformly bounded with respect to $\varepsilon$. For the subsonic flow, density $\rho^{\varepsilon}$ can be represented as a function of   $|u^\varepsilon|^2$, \emph{i.e.},
\begin{equation}
\rho^\varepsilon=\rho^\varepsilon(|u^\varepsilon|^2, \phi)
=\tilde{h}^{-1}\left(\frac{\varepsilon^2\left(2\phi-|u^\varepsilon|^2\right)}{2}+\tilde{h}(1)\right).\label{2.21}
\end{equation}
For any fixed $0<\varepsilon<1$, when the flow is subsonic,  it holds that
\begin{equation}
\tilde{H}^{-1}\circ\tilde{h}(1)<\tilde{H}^{-1}\left(\tilde{h}(1)+\varepsilon^2 \phi \right) =\rho^\varepsilon_{cr}\leq\rho^\varepsilon(|u^\varepsilon|^2, \phi)\leq\tilde{h}^{-1}\left(\varepsilon^2 \phi +\tilde{h}(1)\right)<\tilde{h}^{-1}\left(\phi^\star+\tilde{h}(1)\right).\nonumber
\end{equation}
%which is uniformly bounded respect to $\varepsilon$.

Finally, the low Mach number limit is the limit process when $\varepsilon\rightarrow0$. For the limit, we expect the compressible Euler flow will converge to the corresponding incompressible Euler flow.
Actually, we have the following result which is the main theorem of the paper on \textbf{Problem $(m)$} for the compressible case and the low Mach number limit.

\begin{theorem}\label{MainT2}
Suppose \eqref{psicondition1} holds.
For any fixed $m>0$, there exists constant $\varepsilon_c$ such that when $0<\varepsilon<\varepsilon_c$ there exists a unique solution $(\rho^\varepsilon, u^\varepsilon, p^\varepsilon)\in (C^{\alpha}(\Omega))^{n+2}$ of \textbf{Problem ($m$)} corresponding to equations \eqref{CISEV} with $M^{\varepsilon}<1$. $M^{\varepsilon}$ varies on $(0,1)$ as $\varepsilon$ varies on $(0,\varepsilon_{c})$. Furthermore, as $\varepsilon\rightarrow0$, we have that
\begin{equation}
\rho^{\varepsilon}=1+O(\varepsilon^2)
\qquad
u^{\varepsilon}=\bar{u}+O(\varepsilon^2)
\qquad
\nabla p^{\varepsilon}=\nabla\bar{p}+O(\varepsilon^2),
\qquad
M^{\varepsilon}=O(\varepsilon)
\end{equation}
where $(1, \bar{u}, \bar{p})$ is the classical solution of \textbf{ Problem ($m$)} corresponding to equations \eqref{CHEB} obtained in Theorem \ref{MainT1}.
\end{theorem}

\begin{remark}
It is noticeable that the condition \eqref{psicondition1} on $\nabla\phi$ is local and without requirement on the decay behaviour at the infinity. %This is different from the airfoil problem studied in \cite{Gu-Wang2}.  
In fact, for the infinitely long  nozzle problem, the far field behaviour of flow can be treated by a quasi-one-dimensional problem, where the key point is the local average estimate.
\end{remark}

\begin{remark}
	It is easy to check the gravity :
	$$
	\phi=g x_i
	$$
	for $i=1, \cdots, n-1$, satisfying the conditions \eqref{psicondition1} on $\phi$. It can also be applied to the electric field.
\end{remark}

\begin{remark}
	In Theorem \ref{MainT2},  the regularity of  $(\rho, u, p)$ is restricted by the regularity of $\phi$. One can lift the regularity of $u$ and $\rho$ by imposing higher regularity conditions on $\phi$.
\end{remark}

\section{Approximate Problems and Variational Approach}

Unlike the airfoil problem in \cite{Dong2}, the asymptotic behaviours of the flow at the inlet and the outlet are different, %Due to $\Omega$ is the unbounded domain,
in order to employ the variational approach, we also need to construct
a series of truncated problems in bounded domains to approximate the \textbf{Problem I1($m$)} and \textbf{Problem C1 ($m$)} which are introduced later.
%with subsonic truncation.
For $L>0$, let
\begin{equation}
\Omega_L=\big\{x\in \Omega\ \big|\ |x_n| < L\big\},\ \ \ S^\pm_L =\Omega\cap\{x_n=\pm L\}.
\end{equation}

Let $H_L$ is a Hilbert space under $H^1$-norm such that
%Here, we need to introduce the following Hilbert space:
%defining
\begin{equation}
H_L = \left\{\varphi\in H^1(\Omega_L) :\ \varphi\big|_{S_L^-} = 0\right\}.
\end{equation}
%Then, $H_L$ is a Hilbert space under $H^1$-norm. 
%The additional
%boundary condition is on $S^-_{L}$ which is   $\frac{\partial\bar{\varphi}_L}{\partial x_n} =\frac{m}{|S_L^+|}$ in the sense of trace.
%And, the condition on $S^+_L$ implies the
%mass flux of the flow remains $m$. 

In this section, we will introduce the approximate problems of \textbf{Problem $(m)$} to the incompressible case and the compressible case, and then obtain the existence of the solutions of the approximate problems by the variational approach.

\subsection{Incompressible potential flow} Let us consider the incompressible case first.
By $\eqref{CHEB}_3$, we can introduce the velocity potential $\bar{\varphi}$ for the incompressible case that 
\begin{equation}
\nabla \bar{\varphi} = \bar{u}.
\end{equation}

Then, \textbf{Problem ($m$)} for the incompressible case becomes:

\textbf{Problem I1 ($m$)}. Let $n\geq 2$. Find function $\bar{\varphi}$ such that
\begin{equation}\label{Ie1}
\begin{cases}
\Delta \bar{\varphi}=0, \quad &x\in\Omega,\\
\frac{\partial\bar{\varphi}}{\partial \textbf{n}}=0, \quad &x\in\partial\Omega,\\
\int_{S_0}\frac{\partial \bar{\varphi}}{\partial \textbf{l}} ds = m,
\end{cases}
\end{equation}
%where   $\Omega \subset R^n$ is an infinitely long nozzle, $m> 0$ is the mass flux passing through the nozzle, $S_0$ is an arbitrary cross section of the nozzle, $\textbf{n}$ and $\textbf{l}$ are the unit outer normals of the domain   and $S_0$.
where $S_0$ is any arbitrary cross section of the nozzle, and $\textbf{n}$ and $\textbf{l}$ are the unit outer normals of the nozzle wall $\partial\Omega$ and $S_0$ respectively.

As said before, we will truncate the domain to introduce the approximated problems of \textbf{Problem I1 $(m)$} in $\Omega_L$ for $L>0$. More precisely,
%\subsection{Incompressible Case}
Let us consider the following truncated problem for the incompressible flow that

\medskip

\textbf{Problem I2 ($m$, $L$)}:  Find a function $\bar{\varphi}_L$ such that,
\begin{equation}\label{eqbf}
\left\{\begin{array}{ll}
\Delta \bar{\varphi}_L=0,\ \ \ \ & x\in\Omega_L,\\
\frac{\partial \bar{\varphi}_L}{\partial  \textbf{n}} = 0,\ \ \ \ &\partial\Omega\cap\partial\Omega_L,\\
\frac{\partial\bar{\varphi}_L}{\partial x_n} =\frac{m}{|S_L^+|},& \text{on}\ \ S_L^+\\
\bar{\varphi}_L= 0,&\text{on} \ \ S_L^-.
\end{array}\right.
\end{equation}

Here $|S_L^+|$ denotes the area of the cross section $S_L^+$. We remark that boundary condition $\eqref{eqbf}_3$ on $S^+_L$ implies that the mass flux of the flow is $m$.
%\noindent Th boundary conditions are the same ones with the compressible flow. 

Now we can introduce the variational approach to solve \textbf{Problem I2 ($m$, $L$)}. %by
%a variational method. 

Let functional $J(\varphi)$ on $H_L$ be defined as
\begin{equation}
J(\varphi) = \frac{1}{2} \int_{\Omega_L}|\nabla\varphi|^2 dx - \frac{m}{|S_L^+|}\int_{S_L^+}\varphi
dx',
\end{equation}
where  $x'=(x_1, x_2, \ldots, x_{n-1})$. Then in order to show the existence of solutions of
\textbf{Problem I2 ($m$, $L$)}, we will solve the following variational
problem:

\textbf{Problem I3  ($m$, $L$)}:  Find a minimizer $\bar{\varphi}_L\in H_L$ such that
\begin{equation}\label{dej}
J(\bar{\varphi}_L) = \min\limits_{\varphi\in H_L} J(\varphi).
\end{equation}

For the minimizer of \textbf{Problem I3 ($m$, $L$)}, we have the following remark.
\begin{remark}
	The minimizer of \textbf{Problem I3 ($m$, $L$)} is a solution of \textbf{Problem I2 ($m$, $L$)}.
\end{remark}
\begin{proof}
	We only need to show that equation \eqref{eqbf} is the Euler-Lagrangian equation of the variation problem. For any
	$t\in \mathbb{R}^+$ and for any $\varphi\in H_L$ and $\eta\in H_L$, it is easy to know that $\varphi+t\eta\in H_L$, so
	\begin{equation}\label{n14}
	%\begin{aligned} %%
	J(\varphi + t\eta) - J(\varphi)
	= \frac{1}{2}\int_{\Omega_L}\left( |\nabla\varphi + t\nabla\eta|^2-|\nabla\varphi|^2\right)dx - \frac{mt}{|S^+_L|}\int_{S^+_L}\eta dx'.
	%\end{aligned}
	\end{equation} 
%	Since $\nabla\varphi,\ \nabla\eta\in L^2(\Omega_L)$,
%	it shows that
Hence
	$$\liminf_{t\rightarrow 0^+}\frac{1}{t}\left(J(\varphi + t\eta) - J(\varphi)\right)\\
	=\int_\ol \nabla\varphi\cdot\nabla\eta dx
	-\f{m}{|\slb|}\int_\slb\eta dx'.
	$$
If $\varphi$ is the minimizer, then for any $\eta\in H_L$, we have that
$$
J(\varphi+t\eta)-J(\varphi)\geq0\qquad\mbox{and}\qquad J(\varphi-t\eta)-J(\varphi)\geq0.
$$
Therefore, for any $\eta\in H_L$,
	\begin{equation}\label{weakform}
	\int_\ol \g\varphi\cdot\g\eta dx -\f{m}{|\slb|}\int_\slb\eta dx' = 0.
	\end{equation}
It means that $\varphi$ is the solution of \textbf{Problem I2 ($m$, $L$)}. %by the integration by part.
\end{proof}

Therefore, in order to show the existence of solutions of \textbf{Problem I2 ($m$, $L$)}, we only need to show the existence of a minimizer of \textbf{Problem I3 ($m$, $L$)}.

For \textbf{Problem I3 ($m$, $L$)}, we have the following theorem:

\begin{theorem}\label{th-ae} %%average estimate
	\textbf{Problem I3  ($m$, $L$)} admits a unique minimizer $\bar\varphi_L\in H_L$. Moreover, it holds that
	\begin{equation}\label{L1}
	\frac{1}{|\ol|}\int_{\ol}|\g\bar\varphi_L|^2dx \leq Cm^2,
	\end{equation}
	where constant $C$ does not depend on $L$.
\end{theorem}

%\begin{remark}\label{remli}
%If we could obtain the uniform bound of $\frac{1}{|\ol|}\int_{\ol}|\g\bar\varphi_L|^2dx$ with respect to $L$, then we will get the uniform bound of $|\g\bar\varphi_L|_{L^\infty}$ with respect to $L$ by the Moser iteration method. That means the quantity $\frac{1}{|\ol|}\int_{\ol}|\g\bar\varphi_L|^2dx$ is equivalent to $|\g\bar\varphi_L|_{L^\infty}$.
%\end{remark}
\begin{proof} We divide the proof into four steps.
	
	\emph{Step 1.} $J(\varphi)$ is coercive in $H_L$.  For any $\varphi\in H_L$, we know that $\varphi|_{S^-_L}=0$ by the definition. So by the H\"{o}lder inequality,
	\begin{equation}\label{n11}
	\left|\int_{S_L^+}\varphi dx'\right| =	\left|\int_{S_L^+}\int_{-L}^{L}\frac{\partial\varphi}{\partial x_n} dx_n dx'\right| \leq C|\ol|^{\f 12}\lb{\g\varphi}
	\end{equation}
	By  the Cauchy inequality,
	\begin{eqnarray}
	J(\varphi) &=&\frac{1}{2} \int_{\ol}|\g\varphi|^2dx - \f{m}{|S_L^+|}\int_{S_L^+}\varphi dx'
	\nonumber\\
	&\geq& \frac{1}{4}\lb{\g\varphi}^2 - 2 C(m, |S^+_L|,|\ol|),\nonumber
	\end{eqnarray}
	which implies $J(\varphi)$ is coercive.

	\emph{Step 2.} The existence of the minimizer $\bar\varphi_L\in H_L$. Let
	$\{\bar{\varphi}_{L, n}\}\subset H_L$ be a minimizer sequence such that, as $n\rightarrow\infty$
	\[
	J(\bar{\varphi}_{L, n})\rightarrow \alpha = \inf\limits_{\varphi\in H_L}J(\varphi)>-\infty.
	\]
	Since
	$J(\varphi)$ is coercive in $H_L$,
	\[
	\int_\ol |\g\bar{\varphi}_{L, n}|^2dx \leq 4J(\bar{\varphi}_{L, n}) + 8 C(m, |S^+_L|,|\ol|)\leq 4J(0) + 8C(m, |S^+_L|,|\ol|)
	= 8C(m, |S^+_L|,|\ol|).
	\]
	Therefore, there exists a subsequence, still denoted by $\{\bar{\varphi}_{L, n}\}$, which converges weakly to a function $\bar{\varphi}_L \in H_L$. And, by the lower semi-continuity, it holds that
	\begin{equation}\label{n12}
	\lb{\nabla\bar\varphi_L}^2 \leq \liminf_{n\ra\infty}\int_\ol
	|\nabla\bar{\varphi}_{L, n}|^2dx \leq 4C(m, |S^+_L|,|\ol|).
	\end{equation}
	%
%	By the lower semi-continuity, it is easy to check that
%	\begin{equation}
%	\int_{\ol} |\nabla\bar{\varphi}_L|^2dx \leq \liminf_{n\ra\infty}\int_\ol
%	|\nabla\bar{\varphi}_{L, n}|^2dx.
%	\end{equation}
	On the other hand, similar to the proof of \eqref{n11}, we have 
	\begin{equation}
	\int_\slb(\bar{\varphi}_{L, n}-\bar\varphi_L)^2 dx' \leq C(\Omega_L)\bigg(\int_\ol|\bar{\varphi}_{L, n}-\bar{\varphi}_L|^2dx\bigg)^{\f
		12}\bigg(\int_\ol|\nabla\bar{\varphi}_{L, n}-\nabla\bar{\varphi}_L|^2dx\bigg)^{\f 12} .
	\end{equation}
	Then, by the $L^2(\Omega_L)$ strong convergence of the sequence of $\{\bar{\varphi}_{L, n}\}$, we have that
	\begin{equation}\label{n13}
	\int_\slb|\bar{\varphi}_{L, n}-\bar{\varphi}_L|dx' \ra 0, \ \ \ \ \text{ as } n\ra \infty.
	\end{equation}
	Therefore, it follows from (\ref{n12}) and (\ref{n13}) that
	$
	J(\bar\varphi_L)\leq\liminf_{n\ra\infty} J(\bar{\varphi}_{L, n}) = \alpha,
	$
	which means
	\begin{equation}
	J(\bar{\varphi}_L) = \min\limits_{\varphi\in H_L} J(\varphi) = \alpha.
	\end{equation}

	\emph{Step 3}.
	The uniqueness of the minimizer. If there are two minimizers $\varphi_1\in H_L$ and $\varphi_2\in H_L$, then we know that
	$$
	J(\varphi_1)=J(\varphi_2)\leq J(\frac{\varphi_1+\varphi_2}{2}).
	$$ 
	So
	\begin{align*}
	0\geq& J(\varphi_1)+J(\varphi_2)-2J(\frac{\varphi_1+\varphi_2}{2})\\
	  =&\frac{1}{2}\int_{\Omega_L}(|\nabla\varphi_1|^2+|\nabla\varphi_2|^2-2|\nabla(\frac{\varphi_1+\varphi_2}{2})|^2)dx\\
	  \geq&\frac{1}{4}\int_{\Omega_L}|\nabla(\varphi_1-\varphi_2)|^2dx.
	\end{align*}
	By the fact that $\varphi_1=\varphi_2=0$ on $S_L^-$, we know that $\varphi_1=\varphi_2$. Therefore, the minimizer is unique.
	
	\emph{Step 4.} We will show \eqref{L1} in this step. By direct computation and \eqref{n11},
	\begin{equation}
	\int_\ol |\nabla\bar{\varphi}_L|^2dx = J(\bar{\varphi}_L) + \f{m}{|\slb|}\int_\slb\bar{\varphi}_L dx'\leq J(0) + \f{m}{|\slb|}\int_\slb\bar{\varphi}_L dx'
	\leq C\f{m}{|\slb|}|\ol|^{\f 12}\lb{\nabla\bar\varphi_L}.
	\end{equation}
	
	That is
	
	\begin{equation}
	\f 1{|\ol|}\int_\ol |\nabla\bar{\varphi}_L|^2dx \leq C\f{m^2}{|\slb|^2} \leq
	C\f{m^2}{S_{min}^2},
	\end{equation}
	where $S_{min}$ is defined to be the minimum of $|\slb|$.
\end{proof}

For the regularity of the solution of \textbf{Problem I2 ($m$, $L$)}, by the standard elliptic estimate (cf. see \cite{Gilbarg-Trudinger}), %and approximation of derivatives by finite differences, one can get the following regularity lemma:
we have the following lemma:
\begin{lemma}\label{lem-3}
	Assume (\ref{H1}) holds, then there are constants $0<\alpha<1$ and $C$ depending on  $\Omega_L$ such that for any solution $\bar{\varphi}_L\in H_L$ of \textbf{Problem I2 ($m$, $L$)}, we have
	$$
	\sup_{x\in\Omega_{L/2}}|\nabla\bar\varphi—_L|\leq C \left(	||\nabla\bar\varphi_L||^2_{L^2(\Omega_L)}\right),
	$$
	and
	$$
	\sup_{x_1,x_2\in\Omega_{L/2}} \frac{|\nabla\bar\varphi_L(x_1)-\nabla\bar\varphi_L(x_2)|}{|x_1-x_2|^\alpha}\leq C \left(	||\nabla\bar\varphi_L||^2_{L^2(\Omega_L)}\right).
	$$
\end{lemma}
We omit the proof since it is standard.

\subsection{Compressible potential flow}
Now let us consider the compressible case. By $\eqref{CISEV}_3$, we can introduce the velocity potential $\varphi^{(\varepsilon)}$ for the compressible case such that 
\begin{equation}
\nabla \varphi ^{(\varepsilon)}= u^{\varepsilon}.
\end{equation}

Then \textbf{Problem ($m$)} for the compressible case becomes 

\textbf{Problem C1 ($m$)}. Let $n\geq 2$. Find function $\varphi^{(\varepsilon)}$ such that
\begin{equation}\label{ce1}
\begin{cases}
\mbox{div} \left(\rho^\varepsilon(|\nabla\varphi^{(\varepsilon)}|^2, \phi)\nabla\varphi^{(\varepsilon)} \right)=0, \quad &x\in\Omega,\\
\frac{\partial\varphi^{(\varepsilon)}}{\partial \textbf{n}}=0, \quad &x\in\partial\Omega,\\
\int_{S_0}\rho^\varepsilon\left(|\nabla\varphi^{(\varepsilon)}|^2, \phi\right)\frac{\partial \varphi^{(\varepsilon)}}{\partial \textbf{l}} ds = m.
\end{cases}
\end{equation}
%where   $\Omega \subset R^n$ is an infinitely long nozzle, $m> 0$ is the mass flux passing through the nozzle, $S_0$is an arbitrary cross section of the nozzle, $\textbf{n}$ and $\textbf{l}$ are the unit outer normals of the domain   and $S_0$, respectively 
By the straightforward computation, $\eqref{ce1}_1$ can be rewritten as
\begin{equation}\label{comell}
\sum_{ij=1}^n a_{ij}^{\varepsilon}\partial_{ij}\varphi^{(\varepsilon)}+\sum_{i=1}^n b^\varepsilon_i\partial_i\varphi^{(\varepsilon)}=0,
\end{equation}  
where
\begin{equation}
a^\varepsilon_{ij}
=\rho^{\varepsilon}\left(\delta_{ij}-\frac{\varepsilon^2\partial_i\varphi^{(\varepsilon)}\partial_j\varphi^{(\varepsilon)}}
{\tilde{p}'(\rho^\varepsilon)}\right),
\end{equation}
and
\begin{equation}
b^\varepsilon_i=\frac{\varepsilon^2\rho^\varepsilon\partial_i\phi}{\tilde{p}'(\rho^\varepsilon)}.
\end{equation}

For $0<\varepsilon<1$ and $M^\varepsilon<\theta<1$, we have that 
\begin{equation}
0<\lambda_1|\xi|^2\leq\sum_{ij=1}^n a^\varepsilon_{ij}\xi_i\xi_j\leq \lambda_2|\xi|^2,
\end{equation}
where constants $\lambda_1$ and $\lambda_2$ do not depend on %are uniformly bounded respect to 
$\varepsilon$.

%\subsection{The subsonic cut-off}

%In this section, we will transfer \textbf{Problem C1 ($m$)} to a second order partial differential problem, and introduce a respective subsonic cut-off.

Since equation \eqref{comell} is nonlinear, and is strictly elliptic if and only if $M^\varepsilon(\phi)<1$. We do not know whether equation \eqref{comell} is elliptic or not before solving it. %However,  a priori estimate $|\nabla\varphi^{(\varepsilon)}|\leq q^\varepsilon_{\theta}$ for some $0<\theta<1$ is missing. So, 
%Therefore we need to introduce the cut-off to truncate the coefficients of equation \eqref{comell}. For some $0<\theta<1$, and $\theta':=\frac{1+\theta}{2}$, define
Therefore we need to introduce the subsonic cut-off to truncate the coefficients of equation \eqref{comell}. For $0<\varepsilon_0<1$ and $0<\theta<1$, we introduce $\mathring{q}^{\varepsilon_0}_\theta(\phi)=\inf_{0<\varepsilon<\varepsilon_0} q^\varepsilon_\theta(\phi)$, and cut-off function on the phase plane
\begin{equation*}
\hat{q}(q^2, \phi)=\begin{cases}
q^2-2\phi\quad
&\mbox{if}~|q|\leq \mathring{q}^{\varepsilon_0}_\theta(\phi),\\
\mbox{monotone smooth function} \quad
&\mbox{if} 
~\mathring{q}^{\varepsilon_0}_\theta(\phi)   \leq|q|\leq \mathring{q}^{\varepsilon_0}_{\frac{\theta+1}{2}}(\phi),\\
\sup_{x\in\Omega}\left(\left(\mathring{q}^{\varepsilon_0}_{\frac{\theta+1}{2}}\right)^2(\phi)-2\phi\right)(x)\quad
&\mbox{if}~|q|\geq \mathring{q}^{\varepsilon_0}_{\frac{\theta+1}{2}}(\phi).
\end{cases}\end{equation*}
%\begin{equation*}
%\hat{q}^{(\varepsilon)}(q^2)=\begin{cases}
%q^2\quad
%&\mbox{if}~|q|\leq q^\varepsilon_\theta,\\
%\mbox{monotone smooth function} \quad
%&\mbox{if} 
%~q^\varepsilon_\theta   \leq|q|\leq q^\varepsilon_{\theta'},\\
%\left(q^\varepsilon_{\theta'}\right)^2\quad
%&\mbox{if}~|q|\geq q^\varepsilon_{\theta'}.
%\end{cases}\end{equation*}
 
Let $\hat{\rho}^{(\varepsilon)}$ satisfy
\begin{equation}\label{3.1}
\frac{\hat{q}(|u|^2, \phi)}{2}+h^{(\varepsilon)}(\hat{\rho}^{(\varepsilon)})-h^{(\varepsilon)}(1)%=\frac{q_\infty^2}{2}+h_\varepsilon(\rho_\infty)
=0,
\end{equation}
which is equivalent to
\begin{equation}
\hat{\rho}^\varepsilon=\hat{\rho}^\varepsilon(|u|^2, \phi)
=\tilde{h}^{-1}\left(\tilde{h}(1)-\frac{\varepsilon^2\hat{q}(|u|^2, \phi)}{2}\right).
\end{equation}
We denote $\hat{\rho}^\varepsilon_\Lambda(\Lambda, \phi):=\frac{\partial}{\partial \Lambda}\hat{\rho}^\varepsilon(\Lambda, \phi)$, and $\hat{\rho}^\varepsilon_\phi(\Lambda, \phi):=\frac{\partial}{\partial \phi}\hat{\rho}^\varepsilon(\Lambda, \phi)$.
%For any fixed $B$, choose $\varepsilon_c\in(0,1)$ small such that $q^\varepsilon_{\theta}>\sqrt{2B}$ for any $\varepsilon\in(0,\varepsilon_c)$. So we have that for any $q$,  $2B-\tilde{q}^{(\varepsilon)}(q^2)\leq |2B-q^2|$.

Then, \textbf{Problem C1 ($m$)} is reformulated into \textbf{Problem C2 ($m$)} as follows.

\textbf{Problem C2 ($m$)}:
Let $n\geq 2$. Find function $\varphi^{(\varepsilon)}$ to satisfy
\begin{equation}\label{equ-modify}
\begin{cases}
\mbox{div} \left(\hat{\rho}^\varepsilon(|\nabla\varphi^{(\varepsilon)}|^2, \phi)\nabla\varphi^{(\varepsilon)} \right)=0, \quad &x\in\Omega,\\
\frac{\partial\varphi^{(\varepsilon)}}{\partial \textbf{n}}=0, \quad &x\in\partial\Omega,\\
\int_{S_0}\hat{\rho}^\varepsilon\left(|\nabla\varphi^{(\varepsilon)}|^2\right)\frac{\partial \varphi^{(\varepsilon)}}{\partial \textbf{l}}ds = m.
\end{cases}
\end{equation}

By the straightforward calculation, we know that  $\eqref{equ-modify}_1$ can be rewritten as
$$
\sum_{i,j=1}^n\hat{a}_{ij}(\nabla \varphi^{(\varepsilon)}, \phi)\partial_{ij}\varphi^{(\varepsilon)}+\sum_{i=1}^n\hat{b}_i(\nabla\varphi^{(\varepsilon)}, \phi)\partial_i\varphi^{(\varepsilon)}=0,
$$ 
where
\begin{eqnarray}\label{li-3.2}
\hat{a}_{ij}\left(\nabla \varphi^{(\varepsilon)}, \phi\right) &=& \hat{\rho}^{\varepsilon}\left(|\nabla\varphi^{(\varepsilon)}|^2, \phi\right)\left(\delta_{ij} -
\frac{\hat{q}_\Lambda(|\nabla\varphi^{(\varepsilon)}|^2, \phi)\partial_i\varphi^{(\varepsilon)}\partial_j\varphi^{(\varepsilon)}}
{(c^\varepsilon)^2}\right)\nonumber\\
&=&\hat{\rho}^{\varepsilon}\left(|\nabla\varphi^{(\varepsilon)}|^2\right)\left(\delta_{ij} - 
\frac{\varepsilon^{2}\hat{q}_\Lambda(|\nabla\varphi^{(\varepsilon)}|^2, \phi)\partial_i\varphi^{(\varepsilon)}\partial_j\varphi^{(\varepsilon)}}
{\tilde{p}'(\hat{\rho}^{\varepsilon})}\right),
\end{eqnarray}
and
\begin{equation}
\hat{b}_i\left(\nabla \varphi^{(\varepsilon)}, \phi\right)=\frac{\varepsilon^2\hat{\rho}^\varepsilon\hat{q}_\phi(|\nabla\varphi^{(\varepsilon)}|^2, \phi)\partial_i\phi}{\tilde{p}'(\hat{\rho}^\varepsilon)}.
\end{equation}
Obviously, 
\begin{equation}\label{li-3.3}
%\hat{\rho}^{(\varepsilon)} (s^2) + 2\left(\hat{\rho}^{(\varepsilon)}\right)'(s^2)s^2<C, \qquad
\hat{\lambda}_1 |\xi|^2 \leq\sum_{i, j =1}^n \hat{a}_{ij} \left(\nabla\varphi^{(\varepsilon)}, \phi\right)\xi_i\xi_j\leq \hat{\lambda}_2 |\xi|^2, \qquad |\hat{b}_i(\nabla\varphi^{(\varepsilon)}, \phi)|\leq C|\partial_{i}\phi|,
\end{equation}
where constants $C$, $\hat{\lambda}_1$, and $\hat{\lambda}_2$ depend only on the subsonic truncation parameters $\theta$ and $\varepsilon_0$, and do not depend on solution $\varphi^{(\varepsilon)}$.

%\section{The variation Formulation in the bounded domain}

%Due to $\Omega$ is the unbounded domain, we also need to construct
%a series of truncated problems to approximate the \textbf{Problem I1($m$)} and \textbf{Problem \^{C1} ($m$)}.
%with subsonic truncation.
%Let $L>0$ ,
%\begin{equation}
%\Omega_L=\big\{x\in \Omega\ \big|\ |x_n| < L\big\},\ \ \ S^\pm_L =\Omega\cap\{x_n=\pm L\}.
%\end{equation}

%Here, we need to introduce the following Hilbert space:
%defining
%\begin{equation}
%H_L = \left\{\varphi\in H^1(\Omega_L) :\ \varphi\big|_{S_L^-} = 0\right\}.
%\end{equation}
%Then, $H_L$ is a Hilbert space under $H^1$-norm. The additional
%boundary condition is on $S^-_{L}$ which is   $\frac{\partial\bar{\varphi}_L}{\partial x_n} =\frac{m}{|S_L^+|}$ in the sense of trace.
%And, the condition on $S^+_L$ implies the
%mass flux of the flow remains $m$. 

%\subsection{Compressible case}
Next, as in the previous subsection for the incompressible case, we approximate \textbf{Problem \^{C1} ($m$)} for any $L$ sufficiently large, by considering the following approximate problem:

\textbf{Problem C3 ($m$, $L$)}: For any sufficiently large $L>0$, find a function $\varphi^{(\varepsilon)}_L$ such that,
\begin{equation}\label{eqcc}
\left\{\begin{array}{ll}
\mbox{div} \left(\hat{\rho}^{\varepsilon}(|\nabla\varphi_L^{(\varepsilon)}|^2, \phi)\nabla\varphi_L^{(\varepsilon)} \right)=0,\ \ \ \ & x\in\Omega_L,\\
\frac{\partial \varphi_L^{(\varepsilon)}}{\partial  \textbf{n}} = 0,\ \ \ \ &\partial\Omega\cap\partial\Omega_L,\\
\hat{\rho}^{\varepsilon}(|\nabla\varphi_L^{(\varepsilon)}|^2, \phi)\frac{\partial\varphi_L^{(\varepsilon)}}{\partial x_n} =\frac{m}{|S_L^+|},& \text{on}\ \ S_L^+\\
\varphi_L^{(\varepsilon)} = 0,&\text{on} \ \ S_L^-.
\end{array}\right.
\end{equation}

%\noindent Th boundary conditions are the same ones with the incompressible flow. 
Similarly to the incompressible case, we will solve \textbf{Problem C3 ($m$, $L$)} by
a variational approach. 
%Straightforward calculation yields that  $\eqref{eqcc}_1$ can be rewritten as
%$\sum_{i,j=1}^n\hat{a}_{ij}(\nabla \varphi_L^{(\varepsilon)})\partial_{ij}\varphi_L^{(\varepsilon)}=0$, where
%\begin{eqnarray}\label{li-3.2}
%\hat{a}_{ij}\left(\nabla \varphi_L^{(\varepsilon)}\right) &=& \hat{\rho}^{(\varepsilon)}\left(|\nabla\varphi_L^{(\varepsilon)}|^2\right)\left(\delta_{ij} -
%\frac{\left(\hat{q}^{(\varepsilon)}\right)'(|\nabla\varphi_L^{(\varepsilon)}|^2)\partial_i\varphi_L^{(\varepsilon)}\partial_j\varphi_L^{(\varepsilon)}}
%{(c^\varepsilon)^2}\right)\nonumber\\
%&=&\hat{\rho}^{(\varepsilon)}\left(|\nabla\varphi_L^{(\varepsilon)}|^2\right)\left(\delta_{ij} - 
%\frac{\varepsilon^{2}\left(\hat{q}^{(\varepsilon)}\right)'(|\nabla\varphi_L^{(\varepsilon)}|^2)\partial_i\varphi_L^{(\varepsilon)}\partial_j\varphi_L^{(\varepsilon)}}
%{\tilde{p}'(\tilde{\rho}^{(\varepsilon)})}\right). 
%\end{eqnarray}
%
%
%
%Moreover, 
%\begin{equation}\label{li-3.3}
%\tilde{\rho}^{(\varepsilon)} (s^2) + 2\left(\tilde{\rho}^{(\varepsilon)}\right)'(s^2)s^2<C, \qquad
%\alpha |\xi|^2 \leq \tilde{a}_{ij} \left(\nabla\varphi_L^{(\varepsilon)}\right)\xi_i\xi_j\leq \beta |\xi|^2,
%\end{equation}
%where constants $C$, $\alpha$ and $\beta$ depend only on the subsonic truncation parameter $\theta$ and does not depend on solution $\varphi_L^{(\varepsilon)}$. 
%To solve \textbf{Problem C3 ($m$, $L$)}, 
Here, we follow the idea used in \cite{Dong2} to introduce a variational formulation. Denote
$$
G^{(\varepsilon)}(\Lambda, \phi) = \frac{1}{2}\int_0^\Lambda \hat{\rho}^{(\varepsilon)}(\lambda, \phi)d\lambda.
$$

In order to compare the solutions of \textbf{Problem I2 ($m$, $L$)} and \textbf{Problem C3 ($m$, $L$)}, %to the compressible system and the incompressible one,
we introduce 
\begin{equation}\label{li01}
I^{(\varepsilon)}\left(\varphi, \bar{\varphi}_L \right) 
=\varepsilon^{-4}\int_\Omega 
\left[G^{(\varepsilon)}\left(|\nabla\varphi|^2, \phi\right)
-G^{(\varepsilon)}\left(|\nabla\bar{\varphi}_L|^2, \phi\right)
-\nabla\bar{\varphi}_L\cdot\left(\nabla\varphi-\nabla\bar{\varphi}_L\right)
\right]dx.
\end{equation}

Define 
\begin{equation}\tilde{\varphi}=\frac{\varphi-\bar{\varphi}_L}{\varepsilon^2}\qquad\mbox{and}\qquad \tilde{\varphi}_L^{(\varepsilon)}=\frac{\varphi_L^{(\varepsilon)}-\bar{\varphi}_L}{\varepsilon^2}.\label{3.31}
\end{equation} 

%From the standard calculation, that 
Obviously, both $\varphi^{(\varepsilon)}_L$ and $\tilde{\varphi}^{(\varepsilon)}_L$ belong to $H_L$.
Then in order to apply the variational approach to find solutions of \textbf{Problem C3 ($m$, $L$)}, let us consider the following problem that 

\textbf{Problem C4 ($m$, $L$)}: Find a minimizer $\tilde{\varphi}_L^{(\varepsilon)}\in H_L$  such that
\begin{equation}
I^{(\varepsilon)}\left(\bar{\varphi}_L+\varepsilon^2\tilde{\varphi}_L^{(\varepsilon)}, \bar{\varphi}_L\right)
=\min_{\tilde{\varphi}\in H_L} I^{(\varepsilon)}(\bar{\varphi}_L+\varepsilon^{2}\tilde{\varphi}, \bar{\varphi}_L).
\end{equation}

%\begin{remark}\label{rem:3.1}
%The minimer of \textbf{Problem C4 ($m$, $L$)} satisfies \eqref{eqcc}.
%\end{remark}
%\begin{proof}
%Let
%$\eta\in C_c^\infty({\R}^n)$.
%The first variation of $I^{(\varepsilon)}\left(\varphi_L^{(\varepsilon)}, \bar{\varphi}_L\right)$ with
%$\eta$ is
%\begin{eqnarray}
%0&=&\int_\Omega\left[2\left( G^{(\varepsilon)} \right)' \left(|\nabla\varphi_L^{(\varepsilon)} |^2\right)\nabla\varphi^{(\varepsilon)}_L\cdot\nabla\eta-\nabla\bar{\varphi}_L\cdot\nabla\eta\right]dx\nonumber\\
%%+\f{m}{|\slb|}\int_\slb\eta dx' 
%&=&\int_\Omega\left[2\left( G^{(\varepsilon)} \right)' \left(|\nabla\varphi_L^{(\varepsilon)} |^2\right)\nabla\varphi^{(\varepsilon)}_L\cdot\nabla\eta\right]dx
%-\frac{m}{|\slb|}\int_\slb\eta dx'
%\end{eqnarray}
%The second and third terms in the above equality vanish due to the definition of the incompressible flow $\bar{\varphi}_L$. Then, the first variation of
%$I^{(\varepsilon)}\left(\varphi_L^{(\varepsilon)}, \bar{\varphi}_L\right)$ 
%associated with $\eta$ is
%\begin{eqnarray}
%0&=&\int_\Omega 2\left ( G^{(\varepsilon)} \right)' \left(|\nabla\varphi_L^{(\varepsilon)}|^2\right)\nabla\varphi_L^{(\varepsilon)}\cdot\nabla\eta dx-\frac{m}{|\slb|}\int_\slb\eta dx'\nonumber\\
%&=&\int_\Omega \tilde{\rho}^{(\varepsilon)} \left(|\nabla\varphi_L^{(\varepsilon)}|^2\right)\nabla\varphi_L^{(\varepsilon)}\cdot\nabla\eta dx-\frac{m}{|\slb|}\int_\slb\eta dx'\nonumber
%\end{eqnarray}
%Hence this remark follows from the identity above.
%\end{proof}

For \textbf{Problem C4 ($m$, $L$)}, we have the following theorem:

\begin{theorem}\label{th-vc} 
	\textbf{Problem C4 ($m$, $L$)} admits a unique minimizer $\tilde{\varphi}_L^{(\varepsilon)}\in H_L$. Moreover, the minimizer $\tilde{\varphi}_L^{(\varepsilon)}$ satisfies that
	\begin{equation}\label{L2}
	\frac{1}{|\ol|}\int_{\ol}|\g\tilde{\varphi}_L^{(\varepsilon)}|^2dx \leq Cm^2,
	\end{equation}
	where constant $C$ does not depend on $L$.
\end{theorem}

\begin{proof}
The proof is divided into four steps.

	\textbf{Step 1.} $I^{(\varepsilon)}(	\bar{\varphi}_L+\varepsilon^{2}\tilde{\varphi}, \bar{\varphi}_L)=I^{(\varepsilon)}(\varphi, \bar{\varphi}_L)$ is coercive with respect to $\tilde{\varphi}$ on $H_L$, \emph{i.e.}, we will show that
\begin{equation}\label{I-phi}
	I^{(\varepsilon)}(\varphi, \bar{\varphi}_L) \geq \frac{C_1}{2}\int_{\Omega_L} |\nabla\tilde{\varphi}|^2dx - C_2\int_{\Omega_L} |\nabla\bar{\varphi}_L|^2dx.
\end{equation}
	
Let 
\begin{equation}
I^{(\varepsilon)}(\varphi, \bar{\varphi}_L)=I^{(\varepsilon)}_1(\varphi, \bar{\varphi}_L)+I^{(\varepsilon)}_2(\varphi, \bar{\varphi}_L),
\end{equation}
where
	\begin{equation*}
I^{(\varepsilon)}_1(\varphi, \bar{\varphi}_L)%\nonumber\\
	%&=&:
		=\varepsilon^{-4}\int_{\Omega_L}
	\left[G^{(\varepsilon)}\left(|\nabla\varphi|^2, \phi\right)
	-G^{(\varepsilon)}\left(|\nabla\bar{\varphi}_L|^2, \phi\right)
	-2G^{(\varepsilon)}_\Lambda \left(|\nabla\bar{\varphi}_L |^2, \phi\right)\nabla\bar{\varphi}_L\cdot\left(\nabla\varphi-\nabla\bar{\varphi}_L\right)
	\right]dx,%\nonumber
	\end{equation*}
	and
	\begin{equation*}
	I^{(\varepsilon)}_2(\varphi, \bar{\varphi}_L)%\nonumber\\
	=:	\varepsilon^{-4}\int_{\Omega_L }
	\left[
	\left(2G^{(\varepsilon)}_\Lambda\left(|\nabla\bar{\varphi}_L |^2, \phi\right)-1\right)\nabla\bar{\varphi}_L\cdot\left(\nabla\varphi-\nabla\bar{\varphi}_L\right)
	\right]dx.
		\end{equation*}

First, we will show that $I_1^{(\varepsilon)}(\varphi,\bar{\varphi}_L)$ is coercive with respect to $\tilde{\varphi}$ in $H_L$. 

Let $p=(p_1, \cdots, p_n)$, and let $F^{(\varepsilon)}(p)=G^{(\varepsilon)}\left(|p|^2, \phi \right)$. By straightforward computation, we can get that
	\begin{eqnarray}
	&&G^{(\varepsilon)}\left(|\nabla\varphi|^2, \phi\right)-G^{(\varepsilon)}\left(|\nabla\bar{\varphi}|^2, \phi\right)
	-2 G^{(\varepsilon)}_\Lambda\left(|\nabla\bar{\varphi}|^2, \phi\right) \nabla\bar{\varphi}\cdot\left(\nabla\varphi-\nabla\bar{\varphi}\right)\nonumber\\
	&=& F^{(\varepsilon)}(\nabla\varphi) - F^{(\varepsilon)}\left( \nabla\bar{\varphi}\right)
	-\nabla F^{(\varepsilon)} \left( \nabla\bar{\varphi}\right) \cdot \left(\nabla\varphi- \nabla\bar{\varphi}\right)\nonumber\\
	&=&\sum_{i, j=1}^n\int_0^1(1-t)\partial_{p_ip_j} F^{(\varepsilon)}\left(t\nabla\varphi+(1-t) \nabla\bar{\varphi}\right)dt
	\partial_i(\varphi-\bar{\varphi})\partial_j(\varphi-\bar{\varphi}).\nonumber
	\end{eqnarray}
	
	It is easy to check that $\partial_{pp}^2F^{(\varepsilon)}$ is uniformly positive due to the subsonic cut-off.  In fact, we have
	\begin{equation}
	\left(\partial_{pp}^2F^{(\varepsilon)}(p)\right)_{i,j}
%	=\hat {\rho}^{(\varepsilon)}\delta_{ij}	+ 2 (\hat{\rho}^{(\varepsilon)})' p_ip_j
	=\hat{a}_{ij}.
	\end{equation}
	
	From property \eqref{li-3.3}, we get the uniformly positivity of $\partial_{pp}^2F$. As a consequence, we have
	
	\begin{equation}\label{ineq-phi}
	C_1\int_{\Omega_L }|\nabla\tilde{\varphi}|^2dx\leq I^{(\varepsilon)}_1(\varphi, \bar{\varphi}_L) \leq
	\tilde{C}_1\int_{\Omega_L }|\nabla\tilde{\varphi}|^2dx.
	\end{equation}

Now, let us consider $I^{(\varepsilon)}_2(\varphi, \bar{\varphi}_L)$. Note that 
	\begin{eqnarray}
	\left|\varepsilon^{-2}\left( 2G^{(\varepsilon)}_\Lambda \left(|\nabla\bar{\varphi}_L |^2, \phi\right)-1\right)\right|	&=&\left|\varepsilon^{-2}\left( \hat{\rho}^{(\varepsilon)}  \left(|\nabla\bar{\varphi}_L|^2, \phi\right)-1\right)\right|	\nonumber\\
	&=&\left|\varepsilon^{-2}\left( 	\tilde{h}^{-1}\left(\frac{-\varepsilon^2\hat{q}(|\nabla\bar{\varphi}_L|^2, \phi)}{2}+\tilde{h}(1)\right)-1\right)\right|\nonumber\\
	&=&\left|\varepsilon^{-2}\left( 	\tilde{h}^{-1}\left(\frac{-\varepsilon^2\hat{q}(|\nabla\bar{\varphi}_L|^2, \phi)}{2}+\tilde{h}(1)\right)-\tilde{h}^{-1}\left(\tilde{h}(1)\right)\right)\right|\nonumber\\
	&=& \left|\frac{-\hat{q}(|\nabla\bar{\varphi}_L|^2, \phi)}{2}\int_{0}^{1} 	\left( \tilde{h}^{-1} \right)'\left(-t\frac{\varepsilon^2\hat{q}(|\nabla\bar{\varphi}_L|^2, \phi)}{2}+\tilde{h}(1)\right) dt\right|%\nonumber\\
%	&&
\nonumber\\
	&\leq& C,
%	(\|\nabla\bar{\varphi}\|_{L^\infty})\nonumber
	%&\leq& \left|\frac{2B-|\nabla\bar{\varphi}|^2}{2}\right|\left|\int_{0}^{1} 	\left( \tilde{h}^{-1} \right)'\left(t\frac{\varepsilon^2\left(2B-\hat{q}^{(\varepsilon)}(|\nabla\bar{\varphi}|^2)\right)}{2}+\tilde{h}(1)\right) dt\right|%\nonumber\\
	%&& 
	%\nonumber
	\end{eqnarray}
where constant $C$ is independent of $\varepsilon$.	
Then	
		\begin{eqnarray}\label{i2}
	|I^{(\varepsilon)}_2(\varphi, \bar{\varphi}_L)|
	&=&\left|	\varepsilon^{-4}\int_{\Omega_L }
	\left[
	\left(2 G^{(\varepsilon)}_\Lambda \left(|\nabla\bar{\varphi}_L |^2, \phi 
	\right)-1\right)\nabla\bar{\varphi}_L\cdot\left(\nabla\varphi-\nabla\bar{\varphi}_L\right)
	\right]dx\right|,\nonumber\\
	&\leq& C
%	(\|\nabla\bar{\varphi}\|_{L^\infty}) 
\varepsilon^{-2}\int_{\Omega_L } |\nabla\bar{\varphi}_L| |\nabla(\varphi-\bar{\varphi}_L)|
	dx\nonumber\\
	&\leq& C
%	(\|\nabla\bar{\varphi}\|_{L^\infty})
\int_{\Omega_L } |\nabla\bar{\varphi}_L|^2dx+\frac{C_1}{2} \varepsilon^{-4}\int_{\Omega} |\nabla(\varphi-\bar{\varphi}_L)|^2dx\nonumber\\
	&=& C
%	(\|\nabla\bar{\varphi}\|_{L^\infty})
\int_{\Omega_L } |\nabla\bar{\varphi}_L|^2dx+\frac{C_1}{2} \int_{\Omega_L } |\nabla\tilde{\varphi}|^2dx.
	\end{eqnarray}	
	
By \eqref{ineq-phi} and \eqref{i2}, we get
\eqref{I-phi}.
It also implies that $I^{(\varepsilon)}(\varphi, \bar{\varphi}_L)$ is bounded from below, \emph{i.e.}, there exists a constant $C>0$  such that, for all $\tilde{\varphi}\in H_L$,
\begin{equation}
I^{(\varepsilon)}(\varphi,\bar{\varphi}_L)\geq -C.
%(\|\nabla\bar{\varphi}\|_{L^\infty})
\end{equation}

\textbf{Step 2}.
Note that $\Omega_L$ is a bounded domain, so $I^{(\varepsilon)}(\varphi,\bar{\varphi}_L)=I^{(\varepsilon)}(\bar{\varphi}_L+\varepsilon^2\tilde{\varphi},\bar{\varphi}_L)$ is finite for any $\tilde{\varphi}\in H_L$. Moreover, by \eqref{ineq-phi} and \eqref{i2}, we can also show that %the upper bound of $I^{(\varepsilon)}(\varphi, \bar{\varphi}_L)$ as:
\begin{equation}\label{I-phi-1}
I^{(\varepsilon)}(\varphi, \bar{\varphi}_L) \leq \frac{C_1+2\tilde{C}_1}{2}\int_{\Omega_L } |\nabla\tilde{\varphi}|^2dx + C_3\int_{\Omega_L } |\nabla\bar{\varphi}_L|^2dx.
\end{equation}

%It means that $I^{(\varepsilon)}(\varphi,\bar{\varphi}_L)=I^{(\varepsilon)}(\bar{\varphi}_L+\varepsilon^2\tilde{\varphi},\bar{\varphi}_L)$ is finite for any $\tilde{\varphi}\in H_L$. 

\textbf{Step 3.} 	We will prove that $I^{(\varepsilon)}(\varphi, \bar{\varphi}_L)$ is uniformly convex in space $H_L$.

	Note that $I^{(\varepsilon)}_2(\varphi, \bar{\varphi}_L)$ is linear with respect to $\varphi$.  Then for any $\varphi_1,\, \varphi_2\in H_L$, we have that
\begin{eqnarray}
&&I^{(\varepsilon)}\left(\varphi_1,  \bar{\varphi}_L\right) + I^{(\varepsilon)} \left(\varphi_2,  \bar{\varphi}_L\right)
-2I^{(\varepsilon)}\left(\frac{\varphi_1+\varphi_2}{2},  \bar{\varphi}_L\right)\nonumber\\
&=&I_1^{(\varepsilon)}\left(\varphi_1,  \bar{\varphi}_L\right) + I_1^{(\varepsilon)} \left(\varphi_2,  \bar{\varphi}_L\right)
-2I_1^{(\varepsilon)}\left(\frac{\varphi_1+\varphi_2}{2},  \bar{\varphi}_L\right)\nonumber\\
&=&\int_{\Omega_L }F(\nabla\varphi_1) + F(\nabla\varphi_2)-2F\left(\frac{\nabla\varphi_1+\nabla\varphi_2}{2}\right) dx\nonumber\\
&\geq& \frac{C_1}{2} \varepsilon^{-4}\|\nabla\varphi_1-\nabla\varphi_2\|_{L^2}^2\nonumber\\
&=&\frac{C_1}{2} \|\nabla
\tilde{\varphi}_1-\nabla\tilde{\varphi}_2\|_{L^2}^2.\label{e:3.21}
\end{eqnarray}

It is the uniform convexity of $I^{(\varepsilon)}$.
	
\textbf{Step 4.} We are now ready to show the unique existence of the minimizer $\tilde{\varphi}^{(\varepsilon)} \in H_L$ of \textbf{Problem C4 ($m$, $L$)}, which satisfies \eqref{L2}. 	
	
Firstly, we will show the continuity of $I^{(\varepsilon)}(\bar{\varphi}_L+\varepsilon^2\tilde{\varphi},  \bar{\varphi}_L)$ with respect to $\tilde{\varphi}$ in $H_L$.
Let $\tilde{\varphi}_1$ and $\tilde{\varphi}_2$ in $H_L$, correspond to $\varphi_1$ and $\varphi_2$ via \eqref{3.31} respectively. We have 
	\begin{eqnarray}
	I^{(\varepsilon)}(\varphi_1,  \bar{\varphi}_L)-I^{(\varepsilon)}(\varphi_2,  \bar{\varphi}_L)&=&\varepsilon^{-4}\int_{\Omega_L }\left[ \frac{1}{2}\int_{|\nabla\varphi_2|^2}^{|\nabla\varphi_1|^2}\hat{\rho}^{\varepsilon}(\Lambda, \phi)d\Lambda -\nabla\bar{\varphi}_L\cdot\nabla\left(\varphi_1-\varphi_2\right)\right] dx\nonumber\\
	&=&\varepsilon^{-4}\int_{\Omega_L }\left[ \frac{1}{2}\int_{|\nabla\varphi_2|^2}^{|\nabla\varphi_1|^2}\hat{\rho}^{\varepsilon}(\Lambda, \phi)d\Lambda -\hat{\rho}^{\varepsilon}(|\nabla\bar{\varphi}_L|^2)\nabla\bar{\varphi}_L\cdot\nabla\left(\varphi_1-\varphi_2\right)\right] dx\nonumber\\
	&&+\varepsilon^{-2}\int_{\Omega_L }\left[ \varepsilon^{-2}\left(\hat{\rho}^{\varepsilon}(|\nabla\bar{\varphi}_L|^2, \phi)-1\right)\nabla\bar{\varphi}_L\cdot\nabla\left(\varphi_1-\varphi_2\right) \right] dx\nonumber
	\end{eqnarray} 
  
  Similar to the argument as done in \emph{Step 1} to obtain \eqref{ineq-phi} and \eqref{i2}, and by the H\"older inequality, we have:
	\begin{align}
	&\left| I^{(\varepsilon)}(\varphi_1,  \bar{\varphi}_L)-I^{(\varepsilon)}(\varphi_2,  \bar{\varphi}_L) \right| \nonumber\\ 
	\leq& \frac{\tilde{C}_1}{2}\int_{\Omega_L }|\nabla(\tilde{\varphi}_1-\tilde{\varphi}_2)|^2dx+ C\varepsilon^{-2}\int_{\Omega_L }|\nabla\bar{\varphi}_L||\nabla(\varphi_1-\varphi_2)|dx\\ 
	\leq&  C||\nabla\tilde{\varphi}_1-\nabla\tilde{\varphi}_2||_{L^2}^2
	+
	C||\nabla\tilde{\varphi}_1-\nabla\tilde{\varphi}_2||_{L^2}. \nonumber
	\end{align} 

	Then we have proved the continuity of the functional $I^{(\varepsilon)}(\bar{\varphi}_L+\varepsilon^2\tilde{\varphi},  \bar{\varphi}_L)$ with respect to $\tilde{\varphi}$ in $H_L$. Based on it, we can show the existence of the minimizer $\tilde{\varphi}^{(\varepsilon)}$ by the standard compactness argument via selecting a subsequence from the subsequence of $\tilde{\varphi}^{(i)}$, where $I^{(\varepsilon)}(\bar{\varphi}_L+\varepsilon^2\tilde{\varphi}^{(i)},  \bar{\varphi}_L)$ converges to the minimal value of the functional $I^{(\varepsilon)}(\bar{\varphi}_L+\varepsilon^2\tilde{\varphi},  \bar{\varphi}_L)$ in $H_L$. %via applying the continuity of the functional $I^{(\varepsilon)}(\bar{\varphi}_L+\varepsilon^2\tilde{\varphi},  \bar{\varphi}_L)$ with respect to $\tilde{\varphi}$ in $H_L$. 

For the uniqueness, if $\varphi_1$ and $\varphi_2$ are two minimizers such that 
$$
I^{(\varepsilon)}(\varphi_1,\bar{\varphi}_L)=I^{(\varepsilon)}(\varphi_2,\bar{\varphi}_L)\leq I^{(\varepsilon)}(\frac{\varphi_1+\varphi_2}{2},\bar{\varphi}_L).
$$

Then by \eqref{e:3.21}, we know that 
%then for any $\varphi_\ast\in H_L$, we have that
\begin{equation}
0\geq I^{(\varepsilon)}\left(\varphi_1,  \bar{\varphi}_L\right) + I^{(\varepsilon)} \left(\varphi_2,  \bar{\varphi}_L\right)
-2I^{(\varepsilon)}\left(\frac{\varphi_1+\varphi_2}{2},  \bar{\varphi}_L\right)\geq \frac{C_1}{2} \varepsilon^{-4}\|\nabla\varphi_1-\nabla\varphi_2\|_{L^2}^2.
\end{equation}

Therefore, 
\begin{equation}
\nabla\varphi_1-\nabla\varphi_2=0.
\end{equation}

Note that $\tilde{\varphi}_1$, $\tilde{\varphi}_2\in H_L$. So $\tilde{\varphi}_1=\tilde{\varphi}_2$. It means that the minimizer is unique in $H_L$.
	
Finally, replacing $\varphi$ by $\bar{\varphi}_L$, we know that $I^{(\varepsilon)}(\varphi,\bar{\varphi}_L)\leq I^{(\varepsilon)}(\bar{\varphi}_L,\bar{\varphi}_L)$. So \eqref{I-phi} leads to
\begin{equation}
\int_{\Omega_L} |\nabla\tilde{\varphi}|^2dx\leq C\int_{\Omega_L} |\nabla\bar{\varphi}_L|^2dx,
\end{equation}
which leads to \eqref{L2} from \eqref{L1}.
\end{proof}

Next, we will show that the minimizer of \textbf{Problem C4 ($m$, $L$)} is actually a solution of \textbf{Problem C3 ($m$, $L$)}.

\begin{proposition}\label{th-ws} 
The minimizer of \textbf{Problem C4 ($m$, $L$)} satisfies equations \eqref{eqcc}.
\end{proposition}

\begin{proof}
	For any
	$t\in \R_+$ and for any $\eta\in H_L$, we have that $\varphi^{(\varepsilon)}_L+t\eta\in H_L$. Then,
\begin{eqnarray}
0&\leq& \liminf_{t\ra 0^+}\f 1t(I^{(\varepsilon)}(\varphi^{(\varepsilon)}_L + t\eta,\bar{\varphi}_L) - I^{(\varepsilon)}(\varphi,\bar{\varphi}_L))\nonumber\\
&= &\liminf_{t\ra 0^+}\f 1t \varepsilon^{-4}\int_{\Omega_L} 
\left[G^{(\varepsilon )}\left(|\nabla(\varphi^{(\varepsilon)}_L+t\eta)|^2, \phi\right)
-G^{(\varepsilon)}\left(|\nabla\bar{\varphi}_L|^2, \phi\right)
-\nabla\bar{\varphi}_L\cdot\left(\nabla(\varphi^{(\varepsilon)}_L+t\eta)-\nabla\bar{\varphi}_L\right)
\right]dx\nonumber\\
&=&\varepsilon^{-4}\int_\Omega\left[2 G^{(\varepsilon)}_\Lambda \left(|\nabla\varphi_L^{(\varepsilon)} |^2, \phi\right)\nabla\varphi^{(\varepsilon)}_L\cdot\nabla\eta-\nabla\bar{\varphi}_L\cdot\nabla\eta\right]dx
\end{eqnarray}

Note that $\eta$ is arbitrary, so
\begin{eqnarray}
\int_{\Omega_L} \hat{\rho}^{(\varepsilon)}(|\nabla\varphi^{(\varepsilon)}_L|^2, \phi)\nabla\varphi^{(\varepsilon)}_L\cdot\nabla\eta  dx&=&\int_\Omega\left[2 G^{(\varepsilon)}_\Lambda\left(|\nabla\varphi_L^{(\varepsilon)} |^2, \phi\right)\nabla\varphi^{(\varepsilon)}_L\cdot\nabla\eta\right]dx\nonumber\\
&=&\int_{\Omega_L} \nabla\bar{\varphi}_L\cdot\nabla\eta dx\nonumber\\
&=&\frac{m}{|\slb|}\int_\slb\eta dx'.\nonumber
\end{eqnarray}

Then \eqref{eqcc} follows by the integration by part.
\end{proof}

Finally, for the regularity of solution of \textbf{Problem C3 ($m$, $L$)}, by the standard elliptic estimate (cf. see \cite{Gilbarg-Trudinger}), we actually have the following lemma.

\begin{lemma}
	Assume \eqref{H1} holds, then there exist constants $0<\alpha<1$ and $C$ depending on $\Omega_L$ such that for any solution $\varphi_L$ of \textbf{Problem C3 ($m$, $L$)}, we have the same estimates as the ones in Lemma \ref{lem-3} by replacing $\bar{\varphi}_L$ by $\varphi_L$.
\end{lemma}

\section{Existence of Solutions of \textbf{Problem I1 ($m$)} and \textbf{Problem C2 ($m$)}}
In order to pass the limit $L\rightarrow\infty$ to obtain solutions of \textbf{Problem I1 ($m$)} and \textbf{Problem C2 ($m$)} from solutions of \textbf{Problem I1 ($m$, $L$)}  or \textbf{Problem C1 ($m$, $L$)} respectively, we need to derive the uniform estimate of solutions of \textbf{Problem I1 ($m$, $L$)}  or \textbf{Problem C2 ($m$, $L$)} with respect $L$. It is the local average estimate.

For the local average estimate, we need to introduce the local set: For $x_0=(x_0', x_{0,n})\in \O$, let
\[
\O_{(a, b)} :=\{x=(x', x_n)\in \O \ |\ a<x_n<b\},
\]
and define %the average quantities  as:

\begin{equation}
P_{(a, b)}(\varphi)=\frac{1}{|\Omega_{(a, b)}|}\int_{\Omega_{(a, b)}}\varphi(x) dx.%=\aint_{\O_{(a, b)}}\varphi(x) dx.
\end{equation}

First, from the properties of the nozzle \eqref{H1}, we have the following lemma and proposition of the Poincar\'{e} inequality.
\begin{lemma}[Uniform Poincar\'e Inequality]\label{th-upi}
	For any $a\in \R$, $1\leq p
	<\infty$, one has \begin{equation}\label{eq-upi}
	\left\| \varphi (x) -P_{(a, a+1)}(\varphi) \right\|_{L^p(\O_{(a, a+1)})} \leq C\|\nabla \varphi(x)\|_{L^p(\O_{(a,a+1)})}, \end{equation} where
	$C$ is a positive constant depending only on $p$, $\O$, independent
	of $a$.
\end{lemma}

This lemma is Theorem 3 in \cite{Du-Yan-Xin}, so we omit the proof.

Moreover, in $\Omega$, it follows from Lemma \ref{th-upi} that we also have the following Poincar\'{e} type inequality.

\begin{proposition}\label{prop1}
For $a<b$, one can obtain:
\begin{equation}
|P_{[a-1, a]}(\varphi)-P_{[b, b+1]}(\varphi)|\leq C \int_{\O_{(a-1, b+1)}}|\nabla\varphi| dx,
\end{equation}
where constant $C$ only depends on $\Omega$ and does not dependent on $a$ and $b$.
\end{proposition}
%We refer the readers to \cite{Du-Yan-Xin} for the proofs.

Based on the inequalities, we will introduce the local average lemma case by case.

\subsection{The local average lemma for the incompressible case}

Now, let us introduce the local average lemma for the incompressible case firstly.

\begin{lemma}\label{th-lae} 
	There is a uniform $l$ depending only on $\Omega$, such that for any $b-a>l$ and $\O_{(a-1, b+1)}\subset \O_{\frac{L}{2}}$ , the solution $\bar{\varphi}_L$ of equations \eqref{eqbf} satisfies
	\begin{equation}\label{eq-LAE-1}
	\frac{1}{|\Omega_{(a, b)}|}\int_{\O_{(a, b)}}|\nabla\bar{\varphi}_L|^2 dx \leq Cm^2,
	\end{equation}
	where $C$ does not depend on  $L$.
\end{lemma}
\begin{proof}
For any $-\frac {L}{2}<a-1<a<b<b+1<\frac{ L}{2}$, let $\eta$ be a smooth function such that $\eta\in C^\infty (\ol)$, $0\leq\eta\leq 1$, $|\g\eta|\leq 1$, $\eta|_{\O_{(a, b)}}=1$, and $\eta|_{\O-\O_{(a-1, b+1)}}=0$.

Then, define the test function as $\eta^2\hat\fai\in H^1(\ol)$, where
\[
\hat\fai(x) =
\begin{cases}
\bar{\varphi}_L(x) - {P}_{[a-1, a]}(\bar{\varphi}_L),&\qquad x_n\leq a,\\
\bar{\varphi}_L(x)-{P}_{[a-1, a]}(\bar{\varphi}_L) -({P}_{[b, b+1]}(\bar{\varphi}_L)-{P}_{[a-1, a]}(\bar{\varphi}_L))\f{x_n-a}{b-a},&\qquad a\leq x_n\leq b,\\
\bar{\varphi}_L(x) - {P}_{[b, b+1]}(\bar{\varphi}_L),&\qquad x_n\geq b.
\end{cases}
\]

Obviously,
$\eta^2\hat\fai\in H_L$. From equations \eqref{eqbf} and by the integration by part, we have
\[
\int_\ol \g\bar{\varphi}_L\cdot\g(\eta^2\hat\fai)dx = 0.
\]

Note that $\nabla\hat\phi=\nabla\bar{\varphi}_L -\frac{ {P}_{[b, b+1]}(\bar{\varphi}_L) - {P}_{[a-1, a]}(\bar{\varphi}_L)}{b-a}\chi_{(a,b)}(x)\vec{e}_n$, where
$\vec{e}_n=(0, \cdots, 0, 1)$, and $\chi_{(a,b)}(x)$ is the
characteristic function of $\Omega_{(a, b)}$.

Therefore,
\begin{eqnarray}
 &&\int_{\O_{(a-1, b+1)}}\eta^2|\g\bar{\varphi}_L|^2 dx
-\int_{\O_{(a, b)}}\eta^2\f{\p\bar{\varphi}_L}{\p x_n}\left(\f{{P}_{[b, b+1]}(\bar{\varphi}_L)-{P}_{[a-1, a]}(\bar{\varphi}_L)}{b-a}\right)dx\nonumber\\
&=& -2\int_{\O_{(a-1, b+1)}}\eta\g\bar{\varphi}_L\cdot\g\eta\hat\fai
dx.
\end{eqnarray}

Note that for any $a<x_n<b$, the conservation of mass flux \eqref{massfluxcondition} implies that
$$
\int_{S_{x_n}}\f{\p\bar{\varphi}_L}{\p x_n}dx' = m.
$$

Then, we have
\begin{equation*}
\int_{\O_{(a-1, b+1)}}\eta^2|\g\bar{\varphi}_L|^2 dx =
-2\int_{\O_{(a-1, b+1)}} \eta\g\bar{\varphi}_L\cdot\g\eta\hat\fai dx
+({P}_{[b, b+1]}(\bar{\varphi}_L)-{P}_{[a-1, a]}(\bar{\varphi}_L))m.
\end{equation*}

Consequently, by the fact that $\nabla\eta=0$ on $\Omega_{(a, b)}$, we have
\begin{eqnarray}\label{n23-1}
\int_{\O_{(a,b)}}|\g\bar{\varphi}_L|^2dx &\leq& \int_{\O_{(a-1,b+1)}}\eta^2|\g\bar{\varphi}_L|^2 dx\nonumber\\
&\leq& 2\left|\int_{\O_{(a-1, a)}\cup\O_{(b, b+1)}} \eta\g\bar{\varphi}_L\cdot\g\eta\hat\fai dx\right| +\left|{P}_{[b, b+1]}(\bar{\varphi}_L)-{P}_{[a-1, a]}(\bar{\varphi}_L)\right|m\nonumber\\
&\leq& C\left|\int_{\O_{(a-1, a)}\cup\O_{(b, b+1)}} |\g\bar{\varphi}_L||\hat\fai|dx\right| +\left|{P}_{[b, b+1]}(\bar{\varphi}_L)-{P}_{[a-1, a]}(\bar{\varphi}_L)\right|m\nonumber\\
&\leq& C\left(\int_{\O_{(a-1, a)}}|\g\bar{\varphi}_L|^2dx\right)^{\f 12}\left(\int_{\O_{(a-1, a)}}|\bar{\varphi}_L - {P}_{[a-1, a]}(\bar{\varphi}_L)|^2dx\right)^{\f 12}\nonumber\\
&& + C \left(\int_{\O_{(b, b+1)}}|\g\bar{\varphi}_L|^2dx\right)^{\f
	12}\left(\int_{\O_{(b, b+1)}}|\bar{\varphi}_L - {P}_{[b, b+1]}(\bar{\varphi}_L)|^2dx\right)^{\f 12}\nonumber\\
&&+\left|{P}_{[b, b+1]}(\bar{\varphi}_L)-{P}_{[a-1, a]}(\bar{\varphi}_L)\right|m.
\end{eqnarray}

%It follows the uniform Poincar\'{e} Inequality that 
%\[\label{n24}
%\int_{\O_{(a-1, a)}}|\bar{\varphi}_L- {P}_{[a-1, a]}(\bar{\varphi}_L)|^2dx\leq
%C\int_{\O_{(a-1, a)}}|\g\bar{\varphi}_L|^2dx,\ \ \int_{\O_{(b, b+1)}}|\bar{\varphi}_L - {P}_{[b, b+1]}(\bar{\varphi}_L)|^2dx \leq C\int_{\O_{(b, b+1)}}|\g\bar{\varphi}_L|^2dx, \] where $C$ does not
%depend on $a$ and $b$. Substituting above inequalities into \eqref{n23} with \textbf{Proposition \ref{prop1}} yields

Then, by Lemma \ref{th-upi}, Proposition \ref{prop1} and \eqref{n23-1}, we know that
\begin{equation}\label{n25-1} \int_{\O_{(a, b)}}|\g\bar{\varphi}_L|^2 dx \leq
C_1\int_{\O_{(a-1, b+1)} -\O_{(a, b)}}|\g\bar{\varphi}_L|^2dx +
C_2 m\int_{\O_{(a-1,b+1)}}|\g\bar{\varphi}_L|dx
\end{equation}
where $C_1$ and $C_2$ are two uniform positive constants.

On the other hand, for any $\delta\ll 1 $, we have 
\begin{equation}\label{n26-1}  m\int_{\O_{(a-1, b+1)}}|\g\bar{\varphi}_L|dx
\leq \delta \int_{\O_{(a-1, b+1)}}|\g\bar{\varphi}_L|^2dx + \f
{1}{4\delta }|\Omega_{(a-1, b+1)}|m^2.
\end{equation}
Here $|\Omega_{(a-1, b+1)}|$ is the measure of $\Omega_{(a-1, b+1)}$.
Combining (\ref{n25-1}) and (\ref{n26-1}) together, we have
\[
\int_{\O_{(a, b)}}|\g\bar{\varphi}_L|^2 dx \leq
\left(\frac{C_1+C_2\delta }{C_1+1}\right)\int_{\O_{(a-1, b+1)}}|\g\bar{\varphi}_L|^2dx + \frac{C_2}{4\delta (C_1+1)}|\Omega_{(a-1, b+1)}|m^2.\]

Introduce
\begin{equation*}
A_{(a,b)} = \f 1{b-a}\int_{\O_{(a,b)}}|\g\bar{\varphi}_L|^2dx.
\end{equation*}

Then it becomes
\begin{equation}\label{inter1}
A_{(a,b)}
\leq \left(\frac{C_1+C_2\delta }{C_1+1}\right)\f{b-a+2}{b-a}A_{(a-1,b+1)} +  \frac{C_2}{4\delta (C_1+1)}\frac{|\Omega_{(a-1, b+1)}|}{b-a}m^2.
\end{equation}
By choosing constant $\delta$ sufficiently small, one can find the uniform constant $l$ such that, when $b-a>l$,
\begin{equation*}
\left(\frac{C_1+C_2\delta}{C_1+1}\right)\f{b-a+2}{b-a} <\vartheta<1.
\end{equation*}
Hence, \eqref{inter1} becomes %to the iteration inequality:
\[
A_{(a,b)}\leq \vartheta A_{(a-1,b+1)}+ C'm^2.
\]

By Lemma 8.23 in \cite{Gilbarg-Trudinger} and \eqref{L1}, we know that there exists $\alpha>0$ such that 
$$
A_{(a,b)}\leq C\left((\frac{b-a}{2L})^{\alpha}A_{(-L,L)}+m^2\right)\leq C\left((\frac{b-a}{2L})^{\alpha}+1\right)m^2.
$$

If $L>l$, where $l$ is sufficiently large, then \eqref{eq-LAE-1} holds where the constant $C$ does not depend on $L$. 
\end{proof}

Combining Lemma \ref{lem-3}, above lemma leads to 
\begin{lemma}\label{lem-3local}
	Assume (\ref{H1}) holds, then there are constants $0<\alpha<1$ and $C$ independent of $L$,
	 for any $b-a>l$ and $\O_{(a-1, b+1)}\subset \O_{\frac{L}{2}}$,
	such that for any solution $\bar{\varphi}_L\in H_L$ of \textbf{Problem I2 ($m$, $L$)}, 
	$$
	\sup_{x\in\Omega_{(a, b)}}|\nabla\bar\varphi—_L|\leq C m^2,
	$$
	and
	$$
	\sup_{x,y\in\Omega_{(a, b)}} \frac{|\nabla\bar\varphi_L(x)-\nabla\bar\varphi_L(y)|}{|x-y|^\alpha}\leq Cm^2.
	$$
\end{lemma}

\subsection{The local average lemma for the compressible case}
Next, let us consider the average lemma for the compressible case. %$\tilde{\varphi}^{(\varepsilon)}_L$:

\begin{lemma}\label{th-laeca} 
	There is a uniform constant $l$ depending only on $\Omega$, such that for any $b-a>l$ and $\O_{(a-1, b+1)}\subset \O_{\frac{L}{2}}$ , the solution $\tilde{\varphi}_L^{(\varepsilon)}$ of equations (\ref{eqcc}) satisfies
	\begin{equation}\label{eq-ca}
	\frac{1}{|\Omega_{(a, b)}|}\int_{\Omega_{(a, b)}}|\nabla\tilde{\varphi}_L^{(\varepsilon)}|^2 dx \leq Cm^2,
	\end{equation}
	where constant $C$ does not depend on $L$.
\end{lemma}

\begin{proof}
For any $-\f L2<a-1<a<b<b+1<\f L2$, define smooth function $\eta$ such that $\eta\in C^\infty (\ol)$, $0\leq\eta\leq 1$, $|\g\eta|\leq 1$, $\eta|_{\O_{(a, b)}}=1$ and $\eta|_{\O-\O_{(a-1, b+1)}}=0$.
Then as done in the proof of Lemma \ref{th-lae}, we introduce
 \begin{equation*}
\hat\fai(x) =
\begin{cases}
\tilde{\varphi}_L^{(\varepsilon)}(x) - {P}_{[a-1, a]}(\tilde{\varphi}_L^{(\varepsilon)}),&\qquad x_n\leq a,\\
\tilde{\varphi}_L^{(\varepsilon)}(x)-{P}_{[a-1, a]}(\tilde{\varphi}_L^{(\varepsilon)}) -({P}_{[b, b+1]}(\tilde{\varphi}_L^{(\varepsilon)})-{P}_{[a-1, a]}(\tilde{\varphi}_L^{(\varepsilon)}))\frac{x_n-a}{b-a},&\qquad a\leq x_n\leq b,\\
\tilde{\varphi}_L^{(\varepsilon)}(x) - {P}_{[b, b+1]}(\tilde{\varphi}_L^{(\varepsilon)}),&\qquad x_n\geq b.
\end{cases}
\end{equation*}

Then the test function $\eta^2\hat\fai$ satisfies that $\eta^2\hat\fai\in H^1(\ol)$ and
$\left(\eta^2\hat\fai\right)|_{x_n=\pm L} = 0$. 

Therefore
$\eta^2\hat\fai\in H_L$ and
\begin{equation*}
\int_{\Omega_L} \left(\hat{\rho}^{\varepsilon}(|\nabla\varphi_L^{(\varepsilon)}|^2, \phi)\nabla\varphi_L^{(\varepsilon)}-\nabla \bar{\varphi}_L \right)\cdot\g(\eta^2\hat\fai)dx = 0.
\end{equation*}

So
\begin{eqnarray}
&&\int_{\Omega_{(a-1,b+1)}}\eta^2\left(\hat{\rho}^{\varepsilon}(|\nabla\varphi_L^{(\varepsilon)}|^2, \phi)\nabla\varphi_L^{(\varepsilon)}-\nabla \bar{\varphi}_L \right)\cdot\g\hat\fai dx
\nonumber\\& =&
-2\int_{\Omega_{(a-1,b+1)}}\eta\left(\hat{\rho}^{\varepsilon}(|\nabla\varphi_L^{(\varepsilon)}|^2, \phi)\nabla\varphi_L^{(\varepsilon)}-\nabla \bar{\varphi}_L \right)\cdot\g\eta\hat\fai dx.\nonumber
\end{eqnarray}

Note that $\g\hat\fai = \g\tilde{\varphi}_L^{(\varepsilon)} -
\frac{{P}_{[b, b+1]}(\tilde{\varphi}_L^{(\varepsilon)})-{P}_{[a-1, a]}(\tilde{\varphi}_L^{(\varepsilon)})}{b-a}\chi_{a,b}(x)\vec{e}_n$. So,
\begin{eqnarray}
 &&\int_{\Omega_{(a-1,b+1)}}
 \eta^2\left(\hat{\rho}^{\varepsilon}(|\nabla\varphi_L^{(\varepsilon)}|^2, \phi)\nabla\varphi_L^{(\varepsilon)}-\nabla \bar{\varphi}_L \right)\cdot\nabla\tilde{\varphi}^{(\varepsilon)}_L dx\nonumber\\
&&+\int_{\Omega_{(a,b)}}
\eta^2\left(\hat{\rho}^{\varepsilon}(|\nabla\varphi_L^{(\varepsilon)}|^2, \phi)\frac{\partial\varphi_L^{(\varepsilon)}}{\partial x_n}-\frac{\partial \bar{\varphi}_L}{\partial x_n} \right)\frac{{P}_{[b, b+1]}(\tilde{\varphi}_L^{(\varepsilon)})-
{P}_{[a-1, a]}(\tilde{\varphi}_L^{(\varepsilon)})}{a-b} dx\nonumber\\
&=& -2\int_{\Omega_{(a-1,b+1)}}\eta\left(\hat{\rho}^{\varepsilon}(|\nabla\varphi_L^{(\varepsilon)}|^2, \phi)\nabla\varphi_L^{(\varepsilon)}-\nabla \bar{\varphi}_L \right)\cdot\g\eta\hat\fai dx. 	
\end{eqnarray}

Since $\eta =1$ on $\O_{a,b}$ and
$\int_{S_{x_n}}\hat{\rho}^{\varepsilon}(|\nabla\varphi_L^{(\varepsilon)}|^2, \phi)\frac{\partial\varphi_L^{(\varepsilon)}}{\partial x_n}dx' =  m=\int_{S_{x_n}}\frac{\partial \bar{\varphi}_L}{\partial x_n}dx'$, the above identity becomes
 \begin{eqnarray}\label{c18}
&&\int_{\Omega_{(a-1,b+1)}}
\frac{\eta^2}{\varepsilon^2}\left(\hat{\rho}^{\varepsilon}(|\nabla\varphi_L^{(\varepsilon)}|^2, \phi)\nabla\varphi_L^{(\varepsilon)}-\nabla \bar{\varphi}_L \right)\cdot\nabla\tilde{\varphi}^{(\varepsilon)}_L dx \nonumber\\
&=& -\frac{2}{\varepsilon^2}\int_{\Omega_{(a-1,b+1)}}\eta\left(\hat{\rho}^{\varepsilon}(|\nabla\varphi_L^{(\varepsilon)}|^2, \phi)\nabla\varphi_L^{(\varepsilon)}-\nabla \bar{\varphi}_L \right)\cdot\g\eta\hat\fai dx.
\end{eqnarray}

For the left hand side of the identity above,
\begin{eqnarray}
\int_{\Omega_{(a-1,b+1)}}
\frac{\eta^2}{\varepsilon^2}\left(\hat{\rho}^{\varepsilon}(|\nabla\varphi_L^{(\varepsilon)}|^2, \phi)\nabla\varphi_L^{(\varepsilon)}-\nabla \bar{\varphi}_L \right)\cdot\nabla\tilde{\varphi}^{(\varepsilon)}_L dx &=&
\int_{\Omega_{(a-1,b+1)}}
\eta^2\hat{\rho}^{\varepsilon}(|\nabla\varphi_L^{(\varepsilon)}|^2, \phi)|\nabla\tilde{\varphi}^{(\varepsilon)}_L|^2 dx
\nonumber\\
&&+\int_{\Omega_{(a-1,b+1)}}
\eta^2\frac{\hat{\rho}^{\varepsilon}(|\nabla\varphi_L^{(\varepsilon)}|^2, \phi)-1}{\varepsilon^2}\nabla \bar{\varphi}_L \cdot\nabla\tilde{\varphi}^{(\varepsilon)}_L dx.\nonumber
\end{eqnarray}

Because $\frac{\hat{\rho}^{\varepsilon}(|\nabla\varphi_L^{(\varepsilon)}|^2, \phi)-1}{\varepsilon^2}$ is bounded, for arbitrary $\nu>0$,
\begin{equation*}
\left|\int_{\O_{(a-1,b+1)}}
\eta^2\frac{\hat{\rho}^{\varepsilon}(|\nabla\varphi_L^{(\varepsilon)}|^2, \phi)-1}{\varepsilon^2}\nabla \bar{\varphi}_L \cdot\nabla\tilde{\varphi}^{(\varepsilon)}_L dx\right|%\nonumber\\
\leq\nu\int_{\O_{(a-1,b+1)}}|\nabla\tilde{\varphi}^{(\varepsilon)}_L|^2 dx+\frac{C}{\nu}\int_{\O_{(a-1,b+1)}}|\nabla\bar{\varphi}_L|^2 dx
\end{equation*}

For the right hand side of the identity (\ref{c18}) , note that $\nabla\eta=0$ on $\Omega_{[a-1, b+1]}\backslash\Omega_{[a, b]}$, so
\begin{eqnarray}
&&\left|{\varepsilon^{-2}}\int_{\Omega_{(a-1,b+1)}}\eta\left(\hat{\rho}^{\varepsilon}(|\nabla\varphi_L^{(\varepsilon)}|^2, \phi)\nabla\varphi_L^{(\varepsilon)}-\nabla \bar{\varphi}_L \right)\cdot\g\eta\hat\fai dx\right|\nonumber\\
&=&\left|\int_{\Omega_{(a-1,b+1)}}\eta\hat{\rho}^{\varepsilon}(|\nabla\varphi_L^{(\varepsilon)}|^2, \phi)\nabla\tilde{\varphi}_L^{(\varepsilon)} \cdot\g\eta\hat\fai dx +\int_{\Omega_{(a-1,b+1)}}\eta\frac{\hat{\rho}^{\varepsilon}(|\nabla\varphi_L^{(\varepsilon)}|^2, \phi)-1}{\varepsilon^2}\nabla \bar{\varphi}_L \cdot\g\eta\hat\fai dx\right|\nonumber\\
&\leq&C_1\left|\int_{\Omega_{[a-1, b+1]}-\Omega_{[a, b]}} |\nabla\tilde{\varphi}^{(\varepsilon)}_L||\hat{\fai}| dx\right|+C_2\left|\int_{\Omega_{[a-1, b+1]}-\Omega_{[a, b]}} |\nabla\bar{\varphi}_L||\hat{\fai} |dx\right|\nonumber\\
&\leq&C_1\int_{\Omega_{[a-1, b+1]}-\Omega_{[a, b]}} |\nabla\tilde{\varphi}^{(\varepsilon)}_L|^2 dx+C_2\int_{\Omega_{[a-1, b+1]}-\Omega_{[a, b]}} |\nabla\bar{\varphi}_L|^2dx+C_3\int_{\Omega_{[a-1, b+1]}-\Omega_{[a, b]}} |\hat{\fai}|^2dx\nonumber.
\end{eqnarray}

Notice that,
\begin{eqnarray}
\int_{\Omega_{[a-1, b+1]}-\Omega_{[a, b]}} |\hat{\fai}|^2dx&=&\int_{\Omega_{[a-1,a]}} |\hat{\fai}|^2dx+\int_{\Omega_{[b,b+1]}} |\hat{\fai}|^2dx\nonumber\\
&=&\int_{\Omega_{[a-1,a]}} |\tilde{\varphi}_L^{(\varepsilon)}(x) - {P}_{[a-1, a]}(\tilde{\varphi}_L^{(\varepsilon)})|^2dx+\int_{\Omega_{[b,b+1]}} |\tilde{\varphi}_L^{(\varepsilon)}(x) - {P}_{[b, b+1]}(\tilde{\varphi}_L^{(\varepsilon)})|^2dx\nonumber\\
&\leq&C_4\int_{\Omega_{[a-1,a]}} |\nabla\tilde{\varphi}_L^{(\varepsilon)}|^2dx+C_4\int_{\Omega_{[b,b+1]}} |\nabla\tilde{\varphi}_L^{(\varepsilon)}(x) |^2dx\nonumber\\
&=&C_4\int_{\Omega_{[a-1, b+1]}-\Omega_{[a, b]}}|\nabla\tilde{\varphi}_L^{(\varepsilon)}(x) |^2dx.
\end{eqnarray}

Then,  taking $\nu=\frac{\lambda}{2}$, we have  
\begin{eqnarray}
\label{n23}
\frac{\lambda}{2}\int_{\O_{(a,b)}}|\nabla\tilde{\varphi}^{(\varepsilon)}_L|^2 dx&\leq&\int_{\Omega_{(a-1,b+1)}}
\eta^2\hat{\rho}^{\varepsilon}(|\nabla\varphi_L^{(\varepsilon)}|^2)|\nabla\tilde{\varphi}^{(\varepsilon)}_L|^2 dx-\frac{\lambda}{2}\int_{\O_{(a,b)}}|\nabla\tilde{\varphi}^{(\varepsilon)}_L|^2 dx\nonumber\\
&\leq&C_5\int_{\Omega_{(a-1, b+1)}-\O_{(a,b)}}|\nabla\tilde{\varphi}^{(\varepsilon)}_L|^2 dx+C_6\int_{\O_{(a-1,b+1)}}|\nabla\bar{\varphi}_L|^2 dx\nonumber\\
&\leq&C_5\int_{\Omega_{(a-1, b+1)}-\O_{(a,b)}}|\nabla\tilde{\varphi}^{(\varepsilon)}_L|^2 dx+C_7(b-a+2)m^2
\end{eqnarray}
where we have used \eqref{eq-LAE-1} for the last inequality. From \eqref{n23}, we further have that 
\begin{equation}\label{ii1}
\int_{\O_{(a,b)}}|\nabla\tilde{\varphi}^{(\varepsilon)}_L|^2 dx\leq \frac{C_5}{C_5+\frac{\lambda}{2}}\int_{\O_{(a-1,b+1)}}|\nabla\tilde{\varphi}^{(\varepsilon)}_L|^2 dx+\frac{C_7}{C_5+\frac{\lambda}{2}}(b-a+2)m^2
\end{equation}

Set
\begin{equation*}
B_{a,b} = \frac{1}{b-a}\int_{\Omega_{(a,b)}}|\nabla\tilde{\varphi}^{(\varepsilon)}_L|^2dx,
\end{equation*}
It follows from (\ref{ii1}) that
\begin{equation*}
B_{a,b}
\leq\frac{b-a+2}{b-a}\frac{C_5}{C_5+\frac{\lambda}{2}}B_{a-1,b+1} + \frac{C_7}{C_5+\frac{\lambda}{2}}\frac{b-a+2}{b-a}m^2.
\end{equation*}

Now, we take $l<b-a$ such that  $\frac{b-a+2}{b-a}\frac{C_5}{C_5+\frac{\lambda}{2}}<\vartheta_0<1$. Then, for any $-\frac {L}{2}<a-1<a<b<b+1<\frac{ L}{2}$, we have
\begin{equation*}
B_{a,b}\leq\vartheta_0 B_{a-1,b+1} + C_8 m^2,
\end{equation*}
where $C_8$ is uniformly bounded.
Hence, following the same argument in the proof of Lemma \ref{th-lae} to obtain \eqref{eq-LAE-1}, we have \eqref{eq-ca}.
%\begin{equation}
%\frac{1}{|\Omega_{\frac{L}{2}}|}\int_{\Omega_{\frac{L}{2}}}|\nabla\tilde{\varphi}^{(\varepsilon)}_L|^2dx\leq 
%\frac{|\Omega_{L}|}{|\Omega_{\frac{L}{2}}|}
%\frac{1}{|\Omega_{L}|}\int_{\Omega_{L}}|\nabla\tilde{\varphi}^{(\varepsilon)}_L|^2dx\leq C m^2.
%\end{equation}
\end{proof}

\begin{lemma}\label{lem-3localcompressible}
	Assume (\ref{H1}) holds, then there are constants $0<\alpha<1$ and $C$ independent of $L$,
	for any $b-a>l$ and $\O_{(a-1, b+1)}\subset \O_{\frac{L}{2}}$,
	such that for any solution $\tilde{\varphi}_L^{(\varepsilon)}\in H_L$ of \textbf{Problem C4 ($m$, $L$)}, 
	$$
	\sup_{x\in\Omega_{(a, b)}}|\nabla\tilde{\varphi}_L^{(\varepsilon)}|\leq C m^2,
	$$
	and
	$$
	\sup_{x, y\in\Omega_{(a, b)}} \frac{|\nabla\tilde{\varphi}_L^{(\varepsilon)}(x)-\nabla \tilde{\varphi}_L^{(\varepsilon)}(y)|}{|x-y|^\alpha}\leq Cm^2.
	$$
\end{lemma}
In order to prove Lemma \ref{lem-3localcompressible}, we need the following proposition.

\begin{proposition}\label{prop-modify}
	Let $a^l_{ij}$ for $i,j = 1, \dots, n$ be $L^{\infty}$ functions on $B_1$, and $\lambda$ be a positive constant. Assume that
	$$
	\forall ~\xi\in\R^n, ~\lambda|\xi|^2\leq \sum_{i, j=1}^n a^l_{ij}\xi_i\xi_j\leq \lambda^{-1} |\xi|^2,
	~\mbox{and} ~f^l_i\in L^q, ~ q >n.
	$$
	Let $w(x)$ be a function in $H^1$. Let $f_i^l$ be functions in $L^q$ with $q>n$. Suppose
	$$
	\sum_{i,j=1}^n\partial_i\left[a^l_{ij}(x) \partial_jw(x)\right] + \sum_{i=1}^n\partial_if^l_i =0
	$$
	holds in the distribution sense. Then $w(x)$ is H\"{o}lder continuous in $B_{{1}/{2}}$ and there exist two constants $0<\alpha\leq 1$, $k$, depending on $\lambda$ such that
	$$
	\sup_{x\in B_{1/2}}|w(x)|\leq k\left(||w||_{L^2(B_1)}+ ||f_i^l||_{L^q(B_1)} \right),
	$$
	$$
	\sup_{x, y\in B_{1/2}}\frac{|w(x)-w(y)|}{|x-y|^{\alpha}}
	\leq k\left(||w||_{L^2(B_1)}+ ||f_i^l||_{L^q(B_1)} \right).
	$$
\end{proposition}
The proof of this proposition can be found in \cite{Gilbarg-Trudinger} (see Theorem 8.24).

Based on Proposition \ref{prop-modify}, we can show the $C^{1,\alpha}$-regularity of $\nabla\tilde{\varphi}_L^{(\varepsilon)}$, which is Lemma \ref{lem-3localcompressible}.

\begin{proof} 
	Let
	$\Phi=\partial_k\tilde{\varphi}^{(\varepsilon)}_L$, $\bar{\varphi}'=\partial_k\bar{\varphi}_L$ for $k=1,\dots,n$. Then by the straightforward calculation, $\Phi$ satisfies
$$
\sum_{i, j =1}^n\partial_i\left(\hat{a}_{ij}\partial_j\Phi\right)+\sum_{i=1}\partial_{i}(\hat{b}_k\partial_i\varphi^{(\varepsilon)})+\varepsilon^{-2}\sum_{i, j =1}^n\partial_i\left(\hat{a}_{ij}\partial_j\bar{\varphi}'\right)=0.
$$
	Since $\Delta\bar{\varphi}_L=0$, $\Delta\bar{\varphi}'=0$. We can change the equation above to:
$$
\sum_{i, j =1}^n\partial_i\left(\hat{a}_{ij}\partial_j\Phi\right)=-\varepsilon^{-2}\sum_{i, j =1}^n\partial_i\left((\hat{a}_{ij}-\delta_{ij})\partial_j\bar{\varphi}'\right)-\sum_{i=1}\partial_{i}(\hat{b}_k\partial_i\varphi^{(\varepsilon)}).
$$
	Here, we introduce
	\begin{eqnarray}
	f_{ij}&=&\varepsilon^{-2}(\hat{a}_{ij}-\delta_{ij})\nonumber\\
	&=&\frac{\hat{\rho}^{(\varepsilon)}-1}{\varepsilon^2}\delta_{ij}-\varepsilon^{-2}\frac{\hat{q}_\Lambda(|\nabla\varphi^{(\varepsilon)}|^2, \phi)\partial_i\varphi^{(\varepsilon)}\partial_j\varphi^{(\varepsilon)}}
		{(c^\varepsilon)^2},
	\end{eqnarray}
	then, 
	\begin{equation}
	\varepsilon^{-2}\sum_{i, j =1}^n\partial_i\left((\hat{a}_{ij}-\delta_{ij})\partial_j\bar{\varphi}'\right)=\sum_{i=1}^n\partial_i(\sum_{j=1}^nf_{ij}\partial_{j}\bar{\varphi}'),
	\end{equation}
	Now we are going to show the uniform $L^\infty$ estimate of $f_{ij}$. 
	For the first term,
	\begin{eqnarray}
	\frac{\hat{\rho}^{\varepsilon}-1}{\varepsilon^2}&=&\frac{\hat{\rho}^{\varepsilon}\left(|\nabla\varphi^{(\varepsilon)}|^2, \phi\right)-1}{\varepsilon^2}\nonumber\\
	&=&\frac{\tilde{h}^{-1}\left(\frac{-\varepsilon^2\hat{q}(|\nabla\varphi^{(\varepsilon)}|^2, \phi)}{2}+\tilde{h}(1)\right)-1}{\varepsilon^2}\nonumber\\
	&=&\frac{\tilde{h}^{-1}\left(\frac{-\varepsilon^2\hat{q}(|\nabla\varphi^{(\varepsilon)}|^2, \phi)}{2}+\tilde{h}(1)\right)-\tilde{h}^{-1}(h(1))}{\varepsilon^2}\nonumber\\
	&=&\frac{-\hat{q}(|\nabla{\varphi}^{(\varepsilon)}|^2, \phi)}{2}\int_{0}^{1} 	\left( \tilde{h}^{-1} \right)'\left(-t\frac{\varepsilon^2\hat{q}(|\nabla{\varphi}^{(\varepsilon)}|^2, \phi)}{2}+\tilde{h}(1)\right) dt%\nonumber\\
	%&&\quad 
\label{diffrho}
	%&\leq&C
	\end{eqnarray}
	For the second term,
	\begin{equation}
	\varepsilon^{-2}\frac{\hat{q}_\Lambda(|\nabla\varphi^{(\varepsilon)}|^2, \phi)\partial_i\varphi^{(\varepsilon)}\partial_j\varphi^{(\varepsilon)}}
	{(c^\varepsilon)^2}=\frac{\hat{q}_\Lambda(|\nabla\varphi^{(\varepsilon)}|^2, \phi)  \partial_i\varphi^{(\varepsilon)}\partial_j\varphi^{(\varepsilon)}}
	{\tilde{p}'(\rho^\varepsilon)}.%\leq C
	\end{equation}
Therefore, due to the cut-off, we have that the uniform $L^{\infty}$ estimate of $f_{ij}$.

Also, for $i, k=1, \cdots, n$,  
\begin{equation}
|\hat{b}_k\partial_i\varphi^{(\varepsilon)}|\leq C|\partial_k\phi|.
\end{equation}
By $\nabla\phi\in L^q$, we can show $\hat{b}_k\partial_i\varphi^{(\varepsilon)}$ are bounded in $L^q$ for $q>n$.

	By employing Proposition \ref{prop-modify}, Lemma \ref{th-laeca} leads to the interior estimate of $\phi$, which conclude the uniformly $L^\infty$ and  H\"older continuous estimates.
	
	%On the other hand, $\partial_i\bar{\varphi}'=\partial_i\bar{\psi}'$, while $\bar{\psi}'=\partial_k\bar{\psi}$.
	
	%For the regularity, we need to divided to two steps.
	
	Next for the boundary estimate near $\partial\Omega$, one can apply Theorem 8.29 in \cite{Gilbarg-Trudinger} to replace Proposition \ref{prop-modify} to follow the arguments above to show the boundary estimates near $\partial\Omega_L$.

	%The further estimate lifting is standard. For the further detail please reference \cite{Gilbarg-Trudinger}.
\end{proof}

\subsection{Existence of solutions of \textbf{Problem I1 ($m$)} and \textbf{Problem C2 ($m$)}}\label{subsec:existence}
Based on the local average lemmas obtained from the previous subsections, we can pass the limit for $(\bar{\varphi}_L, \tilde{\varphi}^{(\varepsilon)}_L)$ as $L\rightarrow\infty$ to obtain the existence of solutions of \textbf{Problem I1 ($m$)} and \textbf{Problem C2 ($m$)}.

For any fixed suitably
large $L$, according to previous subsections, one can get $H^1$
functions $\bar\varphi_L(x)$ and $\tilde{\varphi}^{(\varepsilon)}_L(x)$, which are the 
weak solution  to \textbf{Problem I2 ($m$, $L$)} and  \textbf{Problem C3 ($m$, $L$)} respectively. Without loss of the generality, $\bar\varphi_L(0)=0$ and $\tilde{\varphi}^{(\varepsilon)}_L(0)=0$. Moreover, Lemma \ref{lem-3local} and Lemma \ref{lem-3localcompressible} showed  $\bar{\varphi}_L\in
C^{1,\a}(\cl\O_{L/2})$ and $\tilde{\varphi}^{(\varepsilon)}_L\in
C^{1,\a}(\cl\O_{L/2})$  with 
\begin{equation*}
\| \bar{\varphi}_L\|_{C^{0,\a}\left(\O_{L/2}\right)}\leq C m^2, \qquad \mbox{and} \qquad\| \tilde{\varphi}^{(\varepsilon)}_L\|_{C^{0,\a}\left(\O_{L/2}\right)}\leq C m^2.
\end{equation*}
Also, for any $\eta\in C^\infty_0(\Omega_L)$, we have 
\begin{equation*}
\int_\Omega \g\bar{\varphi}_L\cdot\g\eta dx=0\qquad\mbox{ and }\qquad\int_{\Omega} \hat{\rho}^{\varepsilon}(|\nabla\varphi^{(\varepsilon)}_L|^2,\phi)\nabla\varphi^{(\varepsilon)}_L\cdot\nabla\eta  dx=0,
\end{equation*}
where $\varphi^{(\varepsilon)}_L=\bar{\varphi}_L+\varepsilon^2\tilde{\varphi}^{(\varepsilon)}_L$. For the mass flux condition,
\begin{equation*}
\int_{S_0}\frac{\partial \bar{\varphi}_L}{\partial \textbf{l}} ds = m\qquad\mbox{and }\qquad \int_{S_0}\hat{\rho}^\varepsilon\left(|\nabla\varphi^{(\varepsilon)}_L|^2, \phi\right)\frac{\partial \varphi^{(\varepsilon)}_L}{\partial \textbf{l}}ds = m,
\end{equation*}
where $S_0$ is an arbitrary cross section of $\Omega_L$,  $\textbf{l}$ are the unit outer normals of  $S_0$.

By a standard diagonal argument, there exist functions  $\bar{\varphi}\in
C^{1,\a}(\O)$  and $\tilde{\varphi}^{(\varepsilon)}\in
C^{1,\a}(\O)$ with the subsequences $\bar{\varphi}_{L_n}$ and $\tilde{\varphi}^{(\varepsilon)}_{L_n}$ such that for any
$K$, for some $\alpha'<\alpha$, $\bar{\varphi}_{L_n}\rightarrow\bar{\varphi}$  and $\tilde{\varphi}^{(\varepsilon)}_{L_n}\rightarrow\tilde{\varphi}^{(\varepsilon)}$ in $C^{1,\a'}
\left(\O_K\right)$ as $n\rightarrow \infty$.
Therefore, $\bar{\varphi}$ is the solution of \textbf{Problem I1($m$)}, and $\varphi^{(\varepsilon)}$ is the solution \textbf{Problem C2 ($m$)} with $\varphi^{(\varepsilon)}=\bar{\varphi}+\varepsilon^2\tilde{\varphi}^{(\varepsilon)}$.

\section{ Uniqueness of Solutions in the Infinity Long Nozzle}

\subsection{Uniqueness of incompressible flow}

In this subsection, we will show the uniqueness of solutions to \textbf{Problem I1 ($m$)}.
\begin{lemma}\label{6.1}
Suppose that $\Omega$  satisfies the assumptions (\ref{H1}), and $\bar\varphi_k, (k = 1, 2)$ are weak solutions to the following problem
	\begin{equation}\label{eqi1}
	\left\{\begin{array}{ll}
	\Delta \bar{\varphi}_k=0,\ \ \ \ & x\in\Omega,\\
	\frac{\partial \bar{\varphi}_k}{\partial  \textbf{n}} = 0,\ \ \ \ &\partial\Omega, \\
	\end{array}\right.
	\end{equation}
	associated with the same incoming mass flux $m$. There exists a constant $\check{C}$ such that $\|\nabla\bar\varphi_k \|_{L^\infty}<\check{C}$. Then
	$$\nabla\bar{\varphi}_1=\nabla\bar\varphi_2 ,\qquad x\in\Omega.$$
\end{lemma}

\begin{proof} 
 %Suppose there are two  different solution of \textbf{Problems 3 ($m$)}: $\bar\varphi_1$ and $\bar\varphi_2$. 
 Set $\check{\varphi} = \bar\varphi_1 - \bar\varphi_2$. Then $\check{\varphi}$
satisfies
\begin{equation}\label{eqid1}
	\left\{\begin{array}{ll}
	\Delta\check{\varphi}=0,\ \ \ \ & x\in\Omega,\\
	\frac{\partial \check{\varphi}}{\partial  \textbf{n}} = 0,\ \ \ \ &\partial\Omega.\\
	\end{array}\right.
	\end{equation}

Let $\eta(x)$ be a $C^\infty_0$ function satisfying:
$\eta|_{\O_{(-L, L)}} = 1$, $ \eta|_{\O-\O_{(-L-1, L+1)}} = 0,$ and $|\nabla \eta|\leq 1$.
And
$$\hat\psi(x)=
\left\{\begin{array}{ll}
\check{\varphi}(x)-\check{\varphi}_L^-,\ \ \ \ \ \ \ \ \ \ \ \ \ & x\in\Omega_{(-L-1, -L)},\\
\check{\varphi}(x)-\check{\varphi}_L^- -\frac{\check{\varphi}_L^{+}-\check{\varphi}_L^-}{2L}(x_n+L),\ \ \ & x\in\Omega_{(-L, L)},\\
\check{\varphi}(x)-\check{\varphi}_L^+, \ &  x\in\Omega_{(L, L+1)},
\end{array}
\right.
$$
where 
$\check{\varphi}_L^- = \frac{1}{|\Omega_{(- L-1, -L)}|}\int_{\Omega_{(- L-1, -L)}}\check{\varphi}(x) dx$, and
$\check{\varphi}_L^+ = \frac{1}{|\Omega_{(L, L+1)}|}\int_{\Omega_{(L, L+1)}}\check{\varphi}(x) dx$.

Multiplying on the both sides of the first
equation in (\ref{eqid1}) by $\eta^2 \hat\psi$, and integrating it over $\Omega_L$, one
obtains
\begin{eqnarray}
\label{ni7}
&&\int_{\Omega_{(-L-1,L+1)}}\eta^2 |\nabla\check{\varphi}|^2
dx -
\frac{\check{\varphi}_L^+ -\check{\varphi}_L^-}{2L}\int_{\Omega_{(-L,L)}}\eta^2\frac{\partial \check{\varphi}}{\partial x_n}dx\nonumber\\
&=& -2\int_{\Omega_{(-L-1,-L)}}\eta (\check{\varphi} -\check{\varphi}_L^-)
\nabla \eta
\nabla \check{\varphi} dx-2\int_{\Omega_{(L,L+1)}}\eta (\check{\varphi} -\check{\varphi}_L^+)\nabla \eta\cdot
\nabla \check{\varphi} dx.
\end{eqnarray}

The second term in (\ref{ni7}) vanishes due to the cancellation of mass flux condition.
Similar to the calculation in Lemma \ref{th-lae},  we have
\begin{equation}\label{n8-1}
\int_{\Omega_{(-L,L)}}|\nabla\check{\varphi}|^2 dx \leq C_3
\int_{\Omega_{(-L-1,-L)}\cup\ \Omega_{(L,L+1)}}|\nabla\check{\varphi}|^2 dx,
\end{equation}
where $C_3$ is independents on $L$.
Then, we can have the iteration inequality:
\begin{equation} \int_{\Omega_{(-L,L)}}|\nabla\check{\varphi}|^2 dx \leq
\bar\vartheta\int_{\Omega_{(-L-1,L+1)}}|\nabla\check{\varphi}|^2 dx,
\end{equation}
where $\frac{C_3}{C_3+1}=:\bar\vartheta<1$ is a uniform constant.
By repeating the previous argument, one can get: for $n\geq1$
\begin{eqnarray}\label{n9}
\int_{\Omega_{(-L,L)}}|\nabla\check{\varphi}|^2 dx &\leq&\bar\vartheta^{n}
\int_{\Omega_{(-L-n,L+n)}}|\nabla\check{\varphi}|^2 dx\nonumber\\
&\leq&\bar\vartheta^{n} C_3\int_{\Omega_{(-L-n-1,-L-n)}\cup\ \Omega_{(L+n,L+n+1)}}|\nabla\check{\varphi}|^2 dx\nonumber\\
&\leq& 2\bar\vartheta^{n} \check{C}|S_{max}|,
\end{eqnarray}
where $|S_{max}|$ denotes the maximal of the cross section of the nozzle, and we used $\|\nabla\bar{\varphi}_k \|_{L^\infty}<\check{C}$, for $k=1,2$.
Taking $n\rightarrow\infty$ in (\ref{n9}) yields
$$
\nabla\check{\varphi}=0\ \ \ \ \mbox{in} \ \ \Omega_{L}.
$$
Then, it yields for  $x\in\Omega$, $\nabla\bar{\varphi}_1=\nabla\bar\varphi_2 $.
\end{proof}

Based on Lemma \ref{6.1}, we can conclude the proof of Theorem \ref{MainT1}.
\begin{proof}[Proof of Theorem \ref{MainT1}]
It is easy to see the existence and uniqueness of a classic solution in Theorem \ref{MainT1} follows directly from subsection \ref{subsec:existence} and Lemma \ref{6.1}. Moreover, it follows from Lemma \ref{lem-3local} that $\bar{u}=\nabla\bar{\varphi}\in(C^{\alpha}(\Omega))^n$.  It is noticeable that $\phi\in W^{1 , q}_{loc}$ for $q>n$, which leads $\phi$ is in the H\"{o}lder space for some $\alpha$. By \eqref{pbar}, $\bar{p}\in C^\alpha(\Omega)$.
\end{proof}

\subsection{Uniqueness of compressible flow $\varphi^{(\varepsilon)}$}

\begin{lemma} %(Uniqueness)
Suppose that $\Omega$ satisfies the assumptions (\ref{H1}), and $\varphi^{(\varepsilon)}_{k}$  ($k=1, 2$) are the solutions of \textbf{Problem C2 ($m$)}, with $\|\nabla\varphi^{(\varepsilon)}_k \|_{L^\infty}<\hat{C}$ . Then, for $x\in \Omega$, $\nabla \varphi^{(\varepsilon)}_ {1}(x) = \nabla \varphi^{(\varepsilon)}_{2}(x)$.
\end{lemma}

\begin{proof} Set $\check{\varphi}^{(\varepsilon)}= \varphi^{(\varepsilon)}_{1} - \varphi^{(\varepsilon)}_{2}$. Then $\check{\varphi}^{(\varepsilon)}$
satisfies
\begin{equation}\label{n5}
\left
 \{\begin{array}{ll}
 \displaystyle{\partial_i(A_{ij}\partial_j\check{\varphi}^{(\varepsilon)})=0},\ \ \ \ \ \ & in\ \ \ \Omega,\\[3mm]
  \displaystyle{\frac{\partial\varphi}{\partial\textbf{n}}=0,}& on\ \ \
  \p\O,\\[3mm]
 \end{array}
 \right.
\end{equation}
where
$$
A_{ij}=\int_1^2
\left(\hat{\rho}^{\varepsilon}(|\nabla\varphi^{(\varepsilon)}_s|^2, \phi)\delta_{ij}+2\hat{\rho}^{\varepsilon}_\Lambda(|\nabla\varphi^{(\varepsilon)}_s|^2, \phi)\partial_i\varphi^{(\varepsilon)}_s\partial_j\varphi^{(\varepsilon)}_s\right) ds,
$$
with $\varphi^{(\varepsilon)}_s = (2-s)\varphi^{(\varepsilon)}_1 + (s-1)\varphi^{(\varepsilon)}_2$ .
Moreover, there exist a positive
constant $\lambda$, such that for any vector $\xi\in
\mathbb{R}^n$
\begin{equation}\label{n6}
\lambda |\xi|^2 < {A}_{ij} \xi_i\xi_j < \lambda^{-1} |\xi|^2.
\end{equation}

Let $\eta(x) = \eta(x_n)$ be a $C^\infty_0$ function satisfying
$$
\eta(x_n) \equiv1\ \  \text{for}\ \ |x_n|\leq L;\ \ \ \ \eta(x_n)
\equiv0\ \ \text{for}\ \ |x_n|\geq L+1, \ \ \ and\ \
|\eta'(x_n)|\leq 1,
$$
 For $L > 0$
$$\hat\phi(x)=
\left\{\begin{array}{ll}
\check{\varphi}^{(\varepsilon)}(x)-\check{\varphi}^{(\varepsilon)}_{L-},\ \ \ \ \ \ \ \ \ \ \ \ \ & x\in\Omega_{(-L-1, -L)},\\
\check{\varphi}^{(\varepsilon)}(x)-\check{\varphi}^{(\varepsilon)}_{L-}-\frac{\check{\varphi}^{(\varepsilon)}_{L+}-\check{\varphi}^{(\varepsilon)}_{L-}}{2L}(x_n+L),\ \ \ & x\in\Omega_{(-L, L)},\\
\check{\varphi}^{(\varepsilon)}(x)-\check{\varphi}^{(\varepsilon)}_{L+}, \ &  x\in\Omega_{(L, L+1)},
 \end{array}
 \right.
$$
where $$\check{\varphi}^{(\varepsilon)}_{L-} = \frac1{|\Omega_{ (-L-1,
-L)}|}\int_{\Omega_{(- L-1, -L)}}\check{\varphi}^{(\varepsilon)}(x) dx, \ \ \
\check{\varphi}^{(\varepsilon)}_{L+} = \frac1{|\Omega_{(L, L+1)}|}\int_{\Omega_{(
L, L+1)}}\check{\varphi}^{(\varepsilon)}(x) dx.
$$
Note that $\nabla\hat\phi=\nabla\check{\varphi}^{(\varepsilon)}
-\frac{\check{\varphi}^{(\varepsilon)}_{L+}-\check{\varphi}^{(\varepsilon)}_{L+}}{2L}\chi_{-L,L}(x_n)\vec{e}_n$,
$\vec{e}_n=(0, \cdots, 0, 1)$. Here,  $\chi_{-L,L}(x_n)$ is the characteristic function of $(-L,L)$.

Multiplying on the both sides of the first
equation in (\ref{n5}) by $\eta^2 \hat\phi$, and integrating it over $\Omega$, one
obtains
\begin{eqnarray}\label{n7}
&&\int_{\Omega_{(-L-1,L+1)}}\eta^2 A_{ij}\partial_i \check{\varphi}^{(\varepsilon)}
\partial_j \check{\varphi}^{(\varepsilon)} dx -
\frac{\check{\varphi}^{(\varepsilon)}_{L+}-\check{\varphi}^{(\varepsilon)}_{L-}}{2L}\int_{\Omega_{(-L,L)}}\eta^2 A_{nj}
\partial_j \check{\varphi}^{(\varepsilon)} dx\nonumber\\
&=& -2\int_{\Omega_{(-L-1,-L)}}\eta (\check{\varphi}^{(\varepsilon)}-\check{\varphi}^{(\varepsilon)}_{L-})
A_{ij}\partial_i\eta
\partial_j\check{\varphi}^{(\varepsilon)} dx\nonumber\\
&&-2\int_{\Omega_{(L,L+1)}}\eta (\check{\varphi}^{(\varepsilon)} -\check{\varphi}^{(\varepsilon)}_{L+})
A_{ij}\partial_i\eta
\partial_j\check{\varphi}^{(\varepsilon)} dx.%\nonumber
\end{eqnarray}
The second integral on the left hand side vanishes. Indeed,
\begin{eqnarray}
&&\int_{\Omega_{(-L,L)}}\eta^2 A_{nj}
\partial_j \check{\varphi}^{(\varepsilon)} dx\nonumber\\
&=&\int_{\Omega_{(-L,L)}}\eta^2(\hat{\rho}^{\varepsilon}(|\nabla\varphi^{(\varepsilon)}_1|^2, \phi)\nabla\varphi^{(\varepsilon)}_1-\hat{\rho}^{\varepsilon}(|\nabla\varphi^{(\varepsilon)}_2|^2, \phi)\nabla\varphi^{(\varepsilon)}_2)\cdot
\vec{e}_ndx\nonumber
%&=&\int^L_{-L}\eta^2(t)\int_{S_t}\left(\rho(|\nabla\varphi_1|^2)\nabla\varphi_1\cdot\vec{e}_n-\rho(|\nabla\varphi_2|^2)\nabla\varphi_2\cdot
%\vec{e}_n\right) dx'dt=0,\nonumber
 \end{eqnarray}
since the two solutions possess the same mass flux $m$.

It follows from (\ref{n7}) and (\ref{n6}) that
\begin{eqnarray}
&&\lambda\int_{\Omega_{(-L,L)}}|\nabla\check{\varphi}^{(\varepsilon)}|^2 dx\nonumber\\
&\leq& 4\lambda^{-1}\int_{\Omega_{(-L-1,-L)}}|\check{\varphi}^{(\varepsilon)}-\check{\varphi}^{(\varepsilon)}_{L-}|
|\nabla\check{\varphi}^{(\varepsilon)}| dx +
4\lambda^{-1}\int_{\Omega_{(L,L+1)}}|\check{\varphi}^{(\varepsilon)}-\check{\varphi}^{(\varepsilon)}_{L+}| |\nabla\check{\varphi}^{(\varepsilon)}|
dx\nonumber\\
&\leq& C\int_{\Omega_{(-L-1,-L)}\cup\ \Omega_{(L,L+1)}}|\nabla\check{\varphi}^{(\varepsilon)}|^2 dx.
 \end{eqnarray}
where $C$ is independent of $L$. We have
\begin{equation}\label{n8}
\lambda\int_{\Omega_{(-L,L)}}|\nabla\check{\varphi}^{(\varepsilon)}|^2 dx\leq C\int_{\Omega_{(-L-1,-L)}\cup\ \Omega_{(L,L+1)}}|\nabla\check{\varphi}^{(\varepsilon)}|^2 dx.
\end{equation}
and
\begin{equation}\label{n8-11}
\int_{\Omega_{(-L,L)}}|\nabla\check{\varphi}^{(\varepsilon)}|^2 dx\leq \bar{\vartheta}\int_{\Omega_{(-L-1,-L)}\cup\ \Omega_{(L,L+1)}}|\nabla\check{\varphi}^{(\varepsilon)}|^2 dx,
\end{equation}
where $\frac{C}{C+\lambda}=:\bar{\vartheta}<1$
By repeating the previous argument, one can get: for $n\geq1$
\begin{eqnarray}\label{n9-1}
\int_{\Omega_{(-L,L)}}|\nabla\check{\varphi}^{(\varepsilon)}|^2 dx &\leq&\bar\vartheta^{n}
\int_{\Omega_{(-L-n,L+n)}}|\nabla\check{\varphi}^{(\varepsilon)}|^2 dx\nonumber\\
&\leq&\bar\vartheta^{n} C_3\int_{\Omega_{(-L-n-1,-L-n)}\cup\ \Omega_{(L+n,L+n+1)}}|\nabla\check{\varphi}^{(\varepsilon)}|^2 dx\nonumber\\
&\leq& 2\bar\vartheta^{n} \hat{C}|S_{max}|,
\end{eqnarray}
where we used $\|\nabla{\varphi}^{(\varepsilon)}_k \|_{L^\infty}<\hat{C}$, for $k=1,2$.
Taking $n\rightarrow\infty$ in (\ref{n9-1}) yields
$$
\nabla\check{\varphi}^{(\varepsilon)}=0\ \ \ \ \mbox{in} \ \ \Omega_{L}.
$$
Then, it yields for  $x\in\Omega$, $\nabla\varphi^{(\varepsilon)}_1=\nabla\varphi^{(\varepsilon)}_2 $.
\end{proof}

\begin{remark}
It is easy to see that when $|\nabla\varphi^{(\varepsilon)}|<\mathring{q}^{\varepsilon_0}_\theta(\phi)$, $\varphi^{(\varepsilon)}$ is the same solution as the one obtained in \cite{Du-Yan-Xin}.
\end{remark}

%\section{Proof of Main Theorems}

%\subsection{Proof of Theorem \ref{MainT1}}
%Theorem \ref{MainT1} follows directly from subsection \ref{subsec:existence} and Lemma \ref{6.1}.

%\section{The limit of low Mach number and Bers' Skill}

\section{Proof of Theorem \ref{MainT2}}
%\textbf{Proof of Theorem \ref{MainT}}:	

In this section, we will conclude the proof of Theorem \ref{MainT2} by showing the solutions of \textbf{Problem C2 ($m$)} are solutions of \textbf{Problem C1 ($m$)}, and then consider the convergent rate of the low Mach number limit.

\begin{proof}
	%Firstly, we will prove the part (1): the subsonic case.
	
	Up to now, we have shown that for a given fixed cut-off parameter $\theta$ and $\varepsilon_0$, there exists a unique solution of \textbf{Problem C2 ($m$)}, which is denoted as $\varphi^{(\varepsilon)}(x;\varepsilon_0, \theta)$. 
	It is noticeable that for a given $\theta\in(0,1)$, if $|\nabla\varphi^{(\varepsilon)}(x;\varepsilon_0, \theta)|<\mathring{q}^{\varepsilon_0}_\theta(\phi)$, then $\varphi^{(\varepsilon)}(x;\varepsilon_0, \theta)$ is the unique solution of \textbf{Problem C2 ($m$)}. Note that
	\begin{equation}\label{4.13}
	|\nabla\varphi^{(\varepsilon)}(x;\varepsilon_0, \theta)|= |\varepsilon^2\nabla\tilde{\varphi}^{(\varepsilon)}(x;\varepsilon_0, \theta)+\nabla\bar{\varphi}(x)|\leq \max|\nabla\bar{\varphi}|+C(\varepsilon_0, \theta) \varepsilon^2.
	\end{equation}
	Hence, there exists $\varepsilon_{0, \theta}\leq \varepsilon_0$ such that  $|\nabla\varphi^{(\varepsilon)}(x;\varepsilon_0, \theta)|<\mathring{q}^{\varepsilon_0}_\theta(\phi)$,  for any $0<\varepsilon<\varepsilon_{0, \theta}$. Form the definition and uniqueness of $\varphi^{(\varepsilon)}$, $\{\varepsilon_{0, \theta}\}$ is a non-decreasing sequence respect to $\theta$ with upper bound $\varepsilon_0$. Then, we introduce $\varepsilon_{0, cr}=\varlimsup_{0<\theta<1}\varepsilon_{0, \theta}$ such that for $0<\varepsilon<\varepsilon_{0, cr}$, there exists a unique  $\tilde{\varphi}(x; \varepsilon_0)$,
	\begin{equation}
	|\nabla \varphi^{(\varepsilon)}(x;\varepsilon_0)|=|\varepsilon^2\nabla\tilde{\varphi}^{(\varepsilon)}(x;\varepsilon_0)+\nabla\bar{\varphi}(x)|< \mathring{q}^{\varepsilon_0}_{cr}(\phi),
	\end{equation} 
	which equals $M^\varepsilon(\phi)<1$.  In this case, the cut-off can be removed such that the solution is a solution of \textbf{Problem C1 ($m$)}.
	After removing the subsonic cut-off, we want to optimise the critical $\varepsilon_{cr}$. 
	For each $0<\varepsilon_0<1$,  there exists an $\varepsilon_{0, cr}$, with $0<\varepsilon_{0, cr}\leq \varepsilon_0<1$. Then, the critical $\varepsilon_c=\sup_{0<\varepsilon<1}\varepsilon_{0, cr}$ satisfies for any $\varepsilon\in(0, \varepsilon)$, $0<M^\varepsilon(\phi)<1$, and $|\nabla\tilde{\varphi}^{(\varepsilon)}|$ is uniform bounded respect $\varepsilon$.
	
	%
	%and that $q^\varepsilon_\theta$ will go to the infinity as $\varepsilon$ goes to zero for any $0<\theta<1$.  Then, for each $\theta$, there exists $\varepsilon_\theta$, such that when $0<\varepsilon\leq \varepsilon_{\theta}$, we have $\mathcal{M}^{\varepsilon}(\bar{\varphi}, \theta)\leq \theta$. In this case, the cut-off can be removed such that the solution is a solution of \textbf{Problem 2}.
	%
	%Now we can apply the Bers' skill based the uniform estimate of $\tilde{\varphi}^{(\varepsilon)}$ \eqref{4.13} and the uniqueness Lemma \ref{6.1}. More precisely, for a sequence of $\{\theta_i\}_{n=1}^{\infty}$, which are strictly increasing between $0$ and $1$, there exists a non-decreasing sequence $\{\varepsilon_{\theta_i}\}_{n=1}^{\infty}$ such that for any $\varepsilon\in (0, \varepsilon_{\theta_i}]$, we have $\mathcal{M}^{\varepsilon_{\theta_i}}(\bar{\varphi} \theta)\leq \theta_i$. Since $\varepsilon<1$, there exists a upper limit of $\{\varepsilon_{\theta_i}\}_{n=1}^{\infty}$, which is written as $\varepsilon_c$, such that for any $\varepsilon\in (0, \varepsilon_{c})$, we have $\mathcal{M}^{\varepsilon}(\bar{\varphi}, \theta)<1$. 
	%Then, by the Bers' skill, based on the uniqueness Lemma \ref{6.1}, we can easily show that $M^{\varepsilon}$ varies on $(0,1)$ as $\varepsilon$ varies on $(0,\varepsilon_{c})$.
	%%The uniqueness of \textbf{Problem 4} is already contained in Lemma \ref{6.1}.
	
Finally, let us consider the convergence rate of the low Mach number limit. Note that $\varphi^{(\varepsilon)}=\bar{\varphi}+\varepsilon^2\tilde{\varphi}^{(\varepsilon)}$ in the H\"{o}lder space, so $\nabla\varphi^{(\varepsilon)}=\nabla\bar{\varphi}+\nabla\tilde{\varphi}^{(\varepsilon)}$, which equals to
\begin{equation}\label{uconergencerate}
u^{\varepsilon}=\bar{u}+\varepsilon^2\tilde{u}^{(\varepsilon)}.
\end{equation}
It is noticeable that $\phi\in W^{1, q}_{loc}$ for $q>n$, so $\phi$ is in some H\"{o}lder space.
Therefore, for the density, by \eqref{2.21} $\rho^\varepsilon\in C^\alpha(\Omega)$. Then $p^\varepsilon\in C^\alpha(\Omega)$. By the straightforward computation like in \eqref{diffrho}, we have   
\begin{equation}\label{rhoconergencerate}
\rho^{\varepsilon}=1+O(\varepsilon^2).
\end{equation} 
%which is the same from the analysis before. 
Consequently, the definition of the Mach number leads to $M^\varepsilon=O(\varepsilon)$.
%{\color{red} Same as $\rho^\varepsilon$, $p^\varepsilon\in C^\alpha(\Omega)$.}
Finally, for the gradients of the pressure, we have
\begin{eqnarray}
\nabla p^\varepsilon-\nabla \bar{p}&=& -\mbox{div}(\rho^\varepsilon u^\varepsilon\otimes u^\varepsilon)+\mbox{div}(\bar{u}\otimes \bar{u})\nonumber\\
&=&\mbox{div}( \bar{u}\otimes \bar{u}-\rho^\varepsilon u^\varepsilon\otimes u^\varepsilon )
\end{eqnarray}
From \eqref{uconergencerate} and \eqref{rhoconergencerate}, we can conclude: in the weak sense, 
\begin{equation}
\nabla p^\varepsilon=\nabla \bar{p}+O(\varepsilon^2).
\end{equation}

	It completes the proof of Theorem \ref{MainT2}.
\end{proof}

\bigskip
\textbf{Ackowledgments}:  The authors would like to thank Professor Song Jiang for evaluable suggestions. 
The research of Mingjie Li is supported by the NSFC Grant No. 11671412. The research of Tian-Yi Wang  was supported in part by the NSFC Grant No. 11601401 and the Fundamental Research Funds for the Central Universities(WUT: 2017 IVA 072 and 2017 IVB 066). The research of Wei Xiang was supported in part by the Grants Council of the HKSAR, China (Project No. CityU 21305215, Project No. CityU 11332916, Project No. CityU 11304817 and Project No. CityU 11303518).


\begin{thebibliography}{99}
	
		\bibitem{alazard1} T. Alazard,  \emph{Incompressible limit of the nonisentropic Euler equations with the solid wall boundary conditions}. {Advances in Differential Equations} \textbf{10(1)} (2005) 19--44.
	
%		\bibitem{Asano}
%	\newblock K. Asano,
%	\newblock \emph{On the incompressible limit of the compressible Euler equations,}
%	\newblock {Japan J. Appl. Math.}, \textbf{4} (1987) 455-488.
	
		\bibitem{Bers1}
		\newblock L. Bers,
		\newblock \emph{An existence theorem in two-dimensional gas dynamics,}
		\newblock {Proc. Symposia Appl. Math.}, \textbf{1} (1949) 41--46.
		
%		\bibitem{Bers2}
%		\newblock L. Bers,
%		\newblock \emph{Boundary value problems for minimal surfaces with singularities at infinity,}
%		\newblock {Trans. Amer. Math. Soc.}, \textbf{70} (1951) 465--491.
		
		\bibitem{Bers3}
		\newblock L. Bers,
		\newblock \emph{Existence and uniqueness of a subsonic flow past a given profile,}
		\newblock {Comm. Pure Appl. Math.}, \textbf{7} (1954) 441--504.
		
		\bibitem{Bers4}
		\newblock L. Bers,
		\newblock \emph{Mathematical Aspects of Subsonic and Transonic
	Gas Dynamics,}
        \newblock John Wiley \& Sons, Inc.: New York; Chapman \& Hall, Ltd.: London, 1958.

\bibitem{DBDL}
\newblock D. Bresch, B. Desjardins, E. Grenier and C.-K. Lin, 
\newblock \emph{Low Mach number limit of viscous polytropic flows: formal asymptotics in the periodic case,}
\newblock {Stud. Appl. Math.}, \textbf{109} (2002) 125-149.

  %      \bibitem{CX}
   %     \newblock C. Chen and C.-J. Xie,
    %    \newblock  \emph{Existence of steady subsonic Euler flows
     %   through infinitely long periodic nozzles,}
      %  \newblock  {J. Diff. Eqs.} \textbf{ 252} (2012), 4315--4331.

   %     \bibitem{CC}
    %    \newblock C. Chen,
     %   \emph{Subsonic non-isentropic ideal gas with large vorticity in nozzles},
      %  \newblock \emph{Math. Method Appl. Sci.},  (2015),
		
		\bibitem{Chen-Du-Xie-Xin}
		\newblock C. Chen, L. Du, C. Xie, and Z.P. Xin,
		\newblock \emph{Two Dimensional Subsonic Euler Flow Past a Wall or a Symmetric Body,}
		\newblock {Arch. Rational Mech. Anal.} {\bf 221(2)} (2016), 559--602.
		
			\bibitem{CCZ} G.-Q. Chen, C. Christoforou, and Y. Zhang, \emph{Continuous dependence of entropy solutions to the Euler equations on the adiabatic exponent and Mach number}. {Arch. Rational Mech. Anal.} {\bf 189(1)} (2008), 97--130.
		
%		\bibitem{Chen6}
%		\newblock G.-Q. Chen, C. M. Dafermos, M. Slemrod, and D.-H. Wang,
%		\newblock \emph{On two-dimensional sonic-subsonic flow,}
%		\newblock  {Commun. Math. Phys.}, \textbf{271} (2007), 635--647.
%		
%		\bibitem{Chen7}
%		\newblock G.-Q. Chen,  H. Frid,
%		\newblock \emph{Divergence-measure fields and hyperbolic conservation laws,}
%		\newblock  {Arch. Rational. Mech. Anal.}, \textbf{147} (1999), 89--118.
		
		\bibitem{Chen-Huang-Wang}
		\newblock G.-Q. Chen, F.-M. Huang, and T.-Y. Wang,
		\newblock \emph{Subsonic-sonic limit of
		approximate solutions to multidimensional steady Euler equations,}
		\newblock {Arch, Rational Mech. Anal.}, 219 (2016), no. 2, 719--740
	
		\bibitem{CHWX} G.-Q. Chen, F.-M. Huang, T.-Y. Wang, and W. Xiang, \emph{Incompressible Limit of Solutions of Multidimensional Steady Compressible Euler Equations}, {Z. Angew. Math. Phys.} (2016) 67--75.
		
		\bibitem{CHWX1} G.-Q. Chen, F.-M. Huang, T.-Y. Wang, and W. Xiang, \emph{Steady Euler Flows with Large Vorticity and Characteristic Discontinuities in Arbitrary Infinitely Long Nozzles}, {preprint arXiv:1712.08605.} (2017)
	
		\bibitem{CDX} J. Cheng, L. Du, and W. Xiang, \emph{Incompressible R\'{e}thy Flows in Two Dimensions}, {SIAM J. Math. Anal.}, 49 (2017), 3427--3475.
	
%		\bibitem{Courant-Friedrichs}
%		\newblock R. Courant and K.O. Friedrichs,
%		\newblock {\it Supersonic Flow and Shock Waves},
%		\newblock Interscience Publishers Inc.: New York, 1948.
		
%			\bibitem{DGLM}
%		\newblock B. Desjardins, E. Grenier, P.-L. Lions and N. Masmoudi,
%		\newblock  \emph{Incompressible limit for solutions of the isentropic Navier-Stokes equations with Dirichlet boundary conditions,} 
%		\newblock {J. Math. Pures Appl.}, \textbf{78} (1999) 461-471.
		
		\bibitem{DWX}
		\newblock X. Deng, T.-Y. Wang, and W. Xiang,
		\newblock \emph{Three-Dimensional Full Euler Flows with Nontrivial Swirl in Axisymmetric Nozzles},
		\newblock {SIAM J. Math. Anal.}, \textbf{60} (2018) 2740--2772.
		
		\bibitem{Dong1}
		\newblock G.-C. Dong,
		\newblock \emph{Nonlinear partial differential equations of second
			order.}
		\newblock  American Mathematical Society, Providence, RI, 1991.
		
		
		\bibitem{Dong2}
		\newblock G.-C. Dong, and B. Ou,
		\newblock \emph{Subsonic flows around a body in space,}
		\newblock {Comm. Partial Differential Equations}, \textbf{18} (1993) 355--379.
		
		\bibitem{Du-Yan-Xin}
		\newblock L. Du, Z.P. Xin and W. Yan,
		\newblock \emph{Subsonic Flows in a Multi-Dimensional Nozzle,}
		\newblock {Arch, Rational Mech. Anal.}, \textbf{201} (2011), 965--1012.
		
		\bibitem{DXX}
		\newblock L. Du, C. Xie and Z. Xin,
		\newblock \emph{Steady subsonic ideal flows through an infinitely long nozzle with large vorticity},
		\newblock {Commun. Math. Phys.}, {\bf 328} (2014), 327--354.
		
		
%		\bibitem{Duan-Luo}
%		\newblock B. Duan and Z. Luo,
%		\newblock \emph{Three-dimensional full Euler flows in axisymmetric nozzles,}
%		\newblock  {J. Diff. Eqs.}, \textbf{ 254} (2013), 2705--2731.
		
%		\bibitem{DL}
%		B. Duan and Z. Luo,
%		\emph{Subsonic non-isentropic Euler flows with large vorticity in axisymmetric nozzles},
%		{J. Math. Anal. Appl.}, {\bf 430} (2015), 1037--1057.
		
		\bibitem{Edin} D. Ebin, \emph{The motion of slightly compressible fluids viewed as a motion with strong constraining force}. {Annals of mathematics}, (1977) 141--200.
		
		%\bibitem{ve}V. Elling. Nonexistence of low-mach irrotational inviscid flows around polygons. \textit{J. Diff. Eqns.}, 262 (2017) 2705--2721.
		
		\bibitem{Finn1}
		\newblock R. Finn, and D. Gilbarg,
		\newblock \emph{Asymptotic behavior and uniqueness of plane subsonic flows,}
		\newblock Comm. Pure Appl. Math., \textbf{10} (1957), 23--63.
		
		\bibitem{Gilbarg1}
		\newblock R. Finn, and D. Gilbarg,
		\newblock \emph{Three-dimensional subsonic flows and asymptotic estimates for elliptic partial differential equations,}
		\newblock Acta Math., \textbf{98} (1957) 265--296.

	\bibitem{Gu-Wang1} X. Gu, and T.-Y. Wang, \emph{On subsonic and subsonic-sonic flows in the infinity long nozzle with general conservatives force}, {Acta Mathematica Scientia} 
	\textbf{37} (2017) 752--767.
	
	\bibitem{Gu-Wang2} X. Gu, and T.-Y. Wang, \emph{On subsonic and subsonic-sonic flows with general conservatives force in exterior domains}, Accepted by {Acta Mathematicae Applicatae Sinica}.
		
			\bibitem{Feireisl-N} 
		\newblock E. Feireisl, and A. Novotny,
		\newblock \emph{Singular Limits in Thermodynamics of Viscous Fluid},
		\newblock Birkh\"{a}user, Basel, 2009.
		
%		\bibitem{Frankl} F. Frankl, and M.  Keldysh, Die,  \emph{\"aussere neumann'she aufgabe f\"ur nichtlineare elliptische differentialgleichungen mit anwendung auf die theorie der flugel im kompressiblen gas}. {Bull. Acad. Sci.} \textbf{12} (1934) 561--697.
				
		
		
		\bibitem{Gilbarg-Trudinger}
		\newblock D. Gilbarg, and N. Trudinger
		\newblock \emph{Elliptic partial differential equations of second order}.
		\newblock Springer-Verlag, New York,1983, second edition.
		
%		\bibitem{Gu-Wang} X. Gu and T.-Y. Wang,
%		\newblock \emph{On subsonic and subsonic-sonic flows with general conservatives force in exterior domains,}
%		\newblock preprint 2015.
		
%			\bibitem{Hu-Wang}
%		\newblock X.-P. Hu and D.-H. Wang,
%		\newblock \emph{Low Mach number limit of viscous compressible magnetohydrodynamic flows,}
%		\newblock {SIAM J. Math. Anal.}, \textbf{41} (2009) 1272-1294.
		
		\bibitem{Huang-Wang-Wang} F.-M. Huang, T.-Y. Wang, and Y. Wang,
		\newblock \emph{On multidimensional sonic-subsonic flow,}
		\newblock Acta Math. Sci. Ser. B, \textbf{31} (2011), 2131--2140.
		
%			\bibitem{Iguchi}
%		\newblock T. Iguchi,
%		\newblock \emph{The incompressible limit and the initial layer of the compressible Euler equation in $\mathbb{R}_+^n$,}
%		\newblock {Math. Methods Appl. Sci.}, \textbf{20} (1997) 945-958.
		
		\bibitem{Isozaki1}
		\newblock H. Isozaki,
		\newblock \emph{Singular limits for the compressible Euler equation in an exterior domain,}
		\newblock {J. Reine Angew. Math.}, \textbf{381} (1987) 1-36.
		
			\bibitem{JJL3}
		\newblock S. Jiang, Q. Ju and F. Li,
		\newblock \emph{Incompressible limit of the non-isentropic ideal magnetohydrodynamic equations,}
		\newblock {SIAM J. Math. Anal.} \textbf{48} (2016), no. 1, 302-319. 
		
		\bibitem{JJLX} S. Jiang, Q. Ju, F. Li, and Z.-P. Xin, \emph{Low Mach number limit for the full compressible magnetohydrodynamic equations with general initial data}. {Advances in Mathematics} \textbf{259} (2014) 384--420.
		
			\bibitem{KM1} S. Klainerman and A. Majda, \emph{Singular perturbations of quasilinear hyperbolic systems with large parameters and the incompressible limit of compressible fluids}, {Comm. Pure Appl. Math.} \textbf{34} (1981) 481--524.
		
		\bibitem{KM2} S. Klainerman and A. Majda, \emph{Compressible and incompressible fluids}, {Comm. Pure Appl. Math.} \textbf{35} (1982), 629--653.
	

		
%		\bibitem{Liu-Yian}
%		\newblock L. Liu and H. Yuan,
%		\newblock \emph{Steady subsonic potential flows through infinite multi-dimensional largely-open nozzles,}
%		\newblock  Calc. Var., \textbf{49} (2014) 1--36.
		
%		\bibitem{LU_68}
%		\newblock O.A. Ladyzhenskaya, and N.N. Ural'tseva,
%		\newblock \emph{Linear and quasilinear elliptic equations,}
%		\newblock  Academic Press, New York, 1968.

		
%		\bibitem{ou2}
%		\newblock G. Lu, and B. Ou,
%		\newblock \emph{A Poincar\'{e} Inequality on $R^n$ and Its Application to Potential Fluid Flows in Space,}
%		\newblock Comm. Appl. Nonlinear Anal., \textbf{12(1)} (2005) 1--24.
		
			
%		\bibitem{ou1}
%		\newblock B. Ou,
%		\newblock \emph{An irrotational and incompressible flow around a body in space,}
%		\newblock J. of PDEs, \textbf{7(2)} (1994) 160--170.
		
%		\bibitem{Payne}
%		\newblock L. E. Payne, and H. F. Weinberger,
%		\newblock \emph{Note on a lemma of Finn and Gilbarg,}
%		\newblock Acta Math., \textbf{98} (1957) 297--299.
%		

	\bibitem{Lions P.-L.01} P.-L. Lions and N. Masmoudi, \emph{Incompressible limit for a viscous compressible fluid}, {J. Math. Pures Appl.} \textbf{77} (1998) 585--627.

%\bibitem{Lions-Masmoudi}P.-L. Lions and N. Masmoudi, \emph{On a free boundary barotropic model,} {Ann. Inst. H. Poincar$\acute{e}$ Anal. Non Lin$\acute{e}$aire},  \textbf{16} (1999) 373--410.

%\bibitem{ou2} G. Lu, and B. Ou, \emph{A Poincare Inequality on $R^n$ and Its Application to Potential Fluid Flows in Space}. {Comm. Appl. Nonlinear Anal.} \textbf{12(1)} (2005) 1--24.

\bibitem{Masmoudi}
\newblock N. Masmoudi,
\newblock \emph{Incompressible, inviscid limit of the compressible Navier-Stokes system,} 
\newblock {Ann. Inst. H. Poincar\'{e} Anal. Non Lin\'{e}aire}, \textbf{18} (2) (2001) 199-224.

	\bibitem{Masmoudi2} N. Masmoudi, Examples of singular limits in hydrodynamics, {\it Handbook of differential equations: evolutionary equations}, {\bf 3}, pp. 195--275, 2007.

\bibitem{MS} G. M\'{e}tivier, and S. Schochet,  \emph{The incompressible limit of the non-isentropic Euler equations}. {Archive for rational mechanics and analysis} \textbf{158(1)} (2001) 61--90.

%\bibitem{MS1}
%\newblock G. M\'{e}tivier and S. Schochet, 
%\newblock \emph{Averaging theorems for conservative systems and the weakly compressible Euler equations,}
%\newblock {J. Differential Equations}, \textbf{187} (2003) 106-183.

%\bibitem{ou1} B. Ou, \emph{An irrotational and incompressible flow around a body in space}. {J. of PDEs} \textbf{7(2)} (1994) 160--170.

%\bibitem{Payne} L. Payne and H. Weinberger, \emph{Note on a lemma of Finn and Gilbarg}. {Acta Math.} \textbf{98} (1957) 297--299.

\bibitem{QX} A. Qu and W. Xiang, \emph{Three-Dimensional Steady Supersonic Euler Flow Past a Concave Cornered Wedge with Lower Pressure at the Downstream}.
{Arch Rational Mech. Anal.}, \textbf{228} (2018) 431--476.

\bibitem{Schiffer} M. Schiffer, \emph{Analytical theory of subsonic and supersonic flows}, {Handbuch der Physik.}, pp. 1--161, Springer Berlin Heidelberg, 1960.




	\bibitem{Schochet} S. Schochet, \emph{The mathematical theory of low Mach number flows}. {ESAIM: Mathematical Modelling and Numerical Analysis} \textbf{39(03)} (2005) 441--458.
	
%		\bibitem{Shiffman1}
%		\newblock M. Shiffman,
%		\newblock \emph{On the existence of subsonic flows of a compressible fluid,}
%		\newblock Proc. Nat. Acad. Sci. U.S.A., \textbf{38}	(1952) 434--438.
		
		\bibitem{Shiffman2}
		\newblock M. Shiffman,
		\newblock \emph{On the existence of subsonic flows of a compressible fluid,}
		\newblock   J. Rational Mech. Anal., \textbf{1} (1952) 605--652.
		
			\bibitem{Ukai} S. Ukai,	\emph{The incompressible limit and the initial layer of the compressible Euler equation}, {J. Math. Kyoto Univ.} \textbf{26} (1986) 323--331.
		
		\bibitem{VD} M. Van Dyke, \textit{Perturbation methods in fluid mechanics} (Vol. 964), New York: Academic Press, 1964.
		
		\bibitem{Xin1}
		\newblock C. Xie and Z. Xin,
		\newblock \emph{Global subsonic and subsonic-sonic flows through infinitely long nozzles,}
		\newblock Indiana Univ. Math. J., \textbf{56} (2007), 2991--3023.
		
		\bibitem{Xin2}
		\newblock C. Xie and Z. Xin,
		\newblock \emph{Global subsonic and subsonic-sonic flows through infinitely long axially symmetric nozzles,}
		\newblock J. Diff. Eqs., {\bf 248} (2010), 2657--2683.
		
		\bibitem{Xin4}
		\newblock C. Xie and Z. Xin,
		\newblock \emph{Existence of global steady subsonic Euler
		flows through infinitely long nozzles,}
		\newblock  SIAM J. Math. Anal. \textbf{42} (2010),
		751--784.
		
\end{thebibliography}
\end{document}